\RequirePackage{fix-cm}
\documentclass[smallcondensed,numbook]{svjour3}
\smartqed
\usepackage{indentfirst}
\usepackage{bm}
\usepackage{amssymb}
\usepackage{stmaryrd}
\usepackage{enumerate}
\usepackage{amsmath}
\usepackage{commath}
\usepackage{upgreek}
\usepackage{hyperref}
\usepackage{cleveref}
\usepackage{graphicx}
\newcommand{\fra}{\forall}
\newcommand{\jump}[1]{\llbracket #1 \rrbracket}
\newcommand{\av}[1]{\{\!\!\{#1\}\!\!\}}
\makeatletter
\newcommand{\tnorm}{\@ifstar\@tnorms\@tnorm}
\newcommand{\@tnorms}[1]{%
	\left|\mkern-1.5mu\left|\mkern-1.5mu\left|
	#1
	\right|\mkern-1.5mu\right|\mkern-1.5mu\right|
}
\newcommand{\@tnorm}[2][]{%
	\mathopen{#1|\mkern-1.5mu#1|\mkern-1.5mu#1|}
	#2
	\mathclose{#1|\mkern-1.5mu#1|\mkern-1.5mu#1|}
}
\makeatother

\newcommand{\jsnorm}[1]{\abs{#1}_{\mathrm{J}}}
\newcommand{\dgnorm}[1]{\norm{#1}_{\mathrm{dg}}}

\newcommand{\fsnorm}[1]{\abs{#1}_{\mathrm{F}}}
\newcommand{\vnorm}[1]{\tnorm{#1}_v}

\newcommand{\pfnorm}[1]{\norm{#1}_p}
\newcommand{\pnorm}[1]{\tnorm{#1}_p}

\newcommand{\gsnorm}[1]{\abs{#1}_{\mathrm{G}}}
\newcommand{\osc}[1]{\mathrm{osc}(#1)}

\newcommand\ip[2]{\langle #1, #2 \rangle}

\usepackage{xspace,color}

\spnewtheorem{assumption}{Assumption}{\bf}{\it}
\crefname{assumption}{Assumption}{Assumptions}
\crefname{theorem}{Theorem}{Theorems}
\crefname{lemma}{Lemma}{Lemmas}
\crefname{corollary}{Corollary}{Corollaries}
\crefname{enumi}{Item}{Items}
\crefname{appendix}{Appendix}{Appendices}
\crefname{section}{Section}{Sections}
\crefname{table}{Table}{Tables}
\crefname{figure}{Figure}{Figures}
\journalname{}
\begin{document}
\title{Analysis of pressure-robust embedded-hybridized discontinuous
  Galerkin methods for the Stokes problem under minimal regularity}
\titlerunning{Analysis of EDG--HDG methods for the Stokes problem under
  minimal regularity}
\author{Aaron~Baier-Reinio \and Sander~Rhebergen \and Garth~N.~Wells}
\institute{
  Aaron~Baier-Reinio \at
  Department of Applied Mathematics, University of Waterloo,
  Waterloo, ON N2L~3G1, Canada \\
  \email{ambaierreinio@uwaterloo.ca} \\
  ORCID:~0000-0002-0542-4342
  \and
  Sander~Rhebergen \at
  Department of Applied Mathematics, University of Waterloo,
  Waterloo, ON N2L~3G1, Canada \\
  \email{srheberg@uwaterloo.ca} \\
  ORCID:~0000-0001-6036-0356
  \and
  Garth~N.~Wells \at
  Department of Engineering, University of Cambridge,
  Trumpington Street, Cambridge CB2 1PZ, United Kingdom \\
  \email{gnw20@cam.ac.uk} \\
  ORCID:~0000-0001-5291-7951
}
\date{Received: date / Accepted: date}
\maketitle
\begin{abstract}
  We present analysis of two lowest-order hybridizable discontinuous
  Galerkin methods for the Stokes problem, while making only minimal
  regularity assumptions on the exact solution.  The methods under
  consideration have previously been shown to produce
  $H(\textrm{div})$-conforming and divergence-free approximate
  velocities.  Using these properties, we derive a priori error
  estimates for the velocity that are independent of the pressure. These
  error estimates, which assume only $H^{1+s}$-regularity of the exact
  velocity fields for any $s \in [0, 1]$, are optimal in a discrete
  energy norm.  Error estimates for the velocity and pressure in the
  $L^2$-norm are also derived in this minimal regularity setting.  Our
  theoretical findings are supported by numerical computations.
\end{abstract}
\section{Introduction}
\label{sec:introduction}

Finite element methods for incompressible flows that are
\emph{pressure-robust} have become increasingly popular. Such methods
produce approximate velocity fields for which the \emph{a priori}
velocity error estimates are independent of the pressure approximation.
Numerous classical inf-sup stable finite elements, such as the MINI
element \cite{arnold1984stable} and Bernardi--Raugel elements
\cite{bernardi1985analysis}, are not  pressure-robust
\cite{john2017divergence}, with the velocity error polluted by the
pressure approximation error scaled by the inverse of the viscosity,
which can be large if the pressure is complicated or the viscosity is
small.

One way of achieving pressure-robustness is by stable mixed methods with
$H(\textrm{div})$-conforming and divergence-free approximate velocities
\cite{john2017divergence}.  Methods with these properties may relax
$H^1$-conformity and use discontinuous velocity approximations, as
constructing $H^1$-conforming and inf-sup stable schemes that are also
divergence-free is difficult \cite{john2017divergence}.  For this
reason, discontinuous Galerkin (DG) methods seem to be a natural
candidate for the construction of pressure-robust schemes.  Several
classes of pressure-robust DG methods that produce
$H(\textrm{div})$-conforming and divergence-free approximate velocities
were introduced in \cite{cockburn2007note,wang2007new}.

A drawback of DG methods is that they are, on a given mesh, typically
computationally more expensive than standard conforming methods.
Hybridized discontinuous Galerkin (HDG) methods were introduced to
improve upon the computational efficiency of DG methods while retaining
their desirable properties \cite{cockburn2009unified}. This is
accomplished by introducing extra degrees of freedom on cell facets
which allow for local cell-wise variables to be eliminated by static
condensation.  Examples of $H(\textrm{div})$-conforming and
divergence-free HDG methods are given in
\cite{cockburn2014divergence,rhebergen2017analysis,rhebergen2020embedded}
for the Stokes problem and in
\cite{fu2019explicit,lehrenfeld2016high,rhebergen2018hybridizable} for
the Navier--Stokes problem.

In this paper we study two closely related lowest-order hybridizable
DG methods for the velocity-pressure formulation of the Stokes
problem.  Both methods produce $H(\textrm{div})$-conforming and
divergence-free approximate velocities, and are therefore
pressure-robust.  The first method is the lowest-order HDG method
analyzed in \cite{rhebergen2017analysis,rhebergen2020embedded}.  The
velocity finite element space for this method consists of
discontinuous piecewise linear functions on cells and facets.  As
discussed in \cite{rhebergen2020embedded}, the computational cost of
this method can be reduced, while maintaining pressure-robustness, by
using a continuous basis for the velocity facet space.  Such an
approach is reminiscent of embedded discontinuous Galerkin (EDG)
methods \cite{guzey2007embedded}.  This leads to the EDG--HDG method of
\cite{rhebergen2020embedded}, the lowest-order formulation of which is
the second method considered in this paper.

In \cite{rhebergen2020embedded}, optimal pressure-robust velocity error
estimates were obtained for the HDG and EDG--HDG methods, assuming
$H^2$-regularity of the exact velocity solution.  Moreover, numerical
experiments in \cite{rhebergen2020embedded} suggest that even when
$H^2$-regularity of the exact velocity solution fails to hold, the
methods remain convergent.  Because of their computational efficiency,
the lowest-order HDG and EDG--HDG methods are appealing for problems
with  minimal regularity.  However, error analysis for this minimal
regularity case has not been developed.  This paper closes this gap by
presenting analysis of these lowest-order methods under
$H^{1+s}$-regularity of the exact velocity solution for any real number
$s \in [0, 1]$.

To put the error analysis of this paper into a broader context, we
review briefly the literature relevant to our analysis.  Numerous
classes of non-pressure-robust nonconforming methods for the Stokes
problem were analyzed under minimal regularity assumptions in
\cite{badia2014error,li2014new}.  Key to the analysis of
\cite{badia2014error,li2014new} is a so-called \emph{enrichment
operator} that maps nonconforming discrete functions to
$H^1$-conforming functions. More recently, by using enrichment
operators that map discretely divergence-free functions to exactly
divergence-free ones, this minimal regularity analysis has been extended to
pressure-robust schemes. This is done in the works of
\cite{moment_div_preserving,linke2019pressure,linke2018quasi,verfurth2019quasi},
which establish quasi-optimal and pressure-robust error estimates for
various finite element methods that achieve pressure-robustness by
modifying the source term in the discrete formulation.
In \cite{linke2018quasi,verfurth2019quasi} modified
Crouzeix--Raviart methods are considered, while \cite{linke2019pressure}
focuses on modified conforming methods and \cite{moment_div_preserving}
on modified DG methods.  Finally, a variety of conforming and
nonconforming pressure-robust methods based on an augmented Lagrangian
formulation have been proposed and analyzed under minimal regularity
assumptions in \cite{kreuzer2020quasi}.

The main contributions of this paper are as follows.  First, we derive
a bound on the consistency error of the lowest-order HDG and EDG--HDG
methods, by means of an enrichment operator of the type considered in
\cite{moment_div_preserving}.  A consequence of the hybridized formulation 
is that our consistency error bound contains a new term not found in previous 
works.  However, we show that it is still possible to obtain optimal 
pressure-robust velocity error estimates in a discrete energy norm.  
Pressure-robust velocity error estimates in the $L^2$-norm are also derived, 
and we conclude our analysis by deriving $L^2$-error bounds for the pressure.

The remainder of this paper is organized as follows.  In
\cref{sec:preliminaries} we introduce the Stokes problem and the methods
to be analyzed, and discuss some preliminary results.  The main analysis
is carried out in \cref{sec:error_analysis}, where we derive our error
estimates for the velocity and the pressure.  In
\cref{sec:numerical_examples} our theoretical findings are illustrated
by numerical examples, and the paper ends with conclusions
in~\cref{sec:conclusions}.

\section{Preliminaries}
\label{sec:preliminaries}

Let $\Omega \subset \mathbb{R}^d$ with $d \in \{2, 3\}$ be a connected
and bounded domain with polyhedral boundary $\partial \Omega$.  The codimension
of $\partial \Omega$ is assumed to be one, but we do not require that $\Omega$
be a Lipschitz domain.  In particular, domains with cracks are allowed.  On a
given set $D \subset \Omega$, we let $(\cdot, \cdot)_D$ denote the
$L^2$-inner-product on $D$ and $\norm{\cdot}_D$ the $L^2$-norm on $D$.  Given
an integer $k \geq 0$, we let $\norm{\cdot}_{k,D}$ and $\abs{\cdot}_{k,D}$
denote the usual $H^k$-norm and $H^k$-semi-norm on $D$, respectively.  If $k >
0$ is not an integer, we let $\norm{\cdot}_{k,D}$ denote the fractional order
$H^k$-norm on $D$ as defined in \cite{dauge_stokesreg1,dauge_stokesreg2}.  In
the following we omit the subscript $D$ in the case of $D = \Omega$.

\subsection{Stokes problem}
\label{subsec:stokes}

Let $f \in L^2(\Omega)^d$ be a prescribed body force and $\nu > 0$ a
given constant kinematic viscosity.  The Stokes problem seeks a
velocity field $u \in H_0^1(\Omega)^d$ and pressure field
$p \in L_0^2(\Omega) 
:= \big\{ q \in L^2(\Omega) : \int_{\Omega} q \dif{x} = 0 \big\}$ 
such that
\begin{subequations}
  \label{eq:stokes_wf}
  \begin{alignat}{2}
    \nu a(u, v) + b(v, p) &= (f, v) \qquad &&\fra v \in H_0^1(\Omega)^d, \\
    b(u, q) &= 0 \qquad &&\fra q \in L_0^2(\Omega),
  \end{alignat}
\end{subequations}
where $a : H_0^1(\Omega)^d \times H_0^1(\Omega)^d \rightarrow \mathbb{R}$
and
$b : H_0^1(\Omega)^d \times L_0^2(\Omega) \rightarrow \mathbb{R}$ 
are the bilinear forms
\begin{equation*}
  a(w, v) := (\nabla w, \nabla v), \qquad
  b(v, q) := -(\nabla \cdot v, q).
\end{equation*}
It is known that \cref{eq:stokes_wf} is well-posed, see
e.g.~\cite[Chapter~4]{ern2013theory}. Furthermore, within the reduced
space $V := \{v \in H_0^1(\Omega)^d : \nabla \cdot v = 0\}$ of divergence-free
functions, the velocity $u \in V$ equivalently satisfies the reduced problem
\begin{equation}
  \label{eq:stokes_wf_reduced}
  \nu a(u, v) = (f, v) \qquad \fra v \in V.
\end{equation}

We introduce the space of weakly divergence-free vector fields with
vanishing normal component on $\partial \Omega$,
\begin{equation}
  L_{\sigma}^2(\Omega) :=
  \{ w \in L^2(\Omega)^d : (w, \nabla \psi) = 0 \ \fra \psi \in H^1(\Omega)
  \},
\label{eqn:Lsigma}
\end{equation}
and note that every vector field $f \in L^2(\Omega)^d$ admits a unique
Helmholtz decomposition~\cite[Theorem~2.1]{linke2018quasi}
\begin{equation*}
  f = \nabla \phi + \mathbb{P}f,
\end{equation*}
where $\phi \in H^1(\Omega) / \mathbb{R}$ and $\mathbb{P}f \in
L_{\sigma}^2(\Omega)$.

The vector field $\mathbb{P}f$ is called the Helmholtz projection of $f$, see
e.g.~\cite[Section~2]{john2017divergence}.  We note that the reduced problem in
\cref{eq:stokes_wf_reduced} is equivalent to
\begin{equation}
  \label{eq:stokes_wf_reduced_helmholtz}
  \nu a(u, v) = (\mathbb{P} f, v) \qquad \fra v \in V,
\end{equation}
since for all $v \in V$ it holds that $(f, v) = (\mathbb{P} f, v)$. In
particular, the velocity solution $u$ is determined only by the Helmholtz
projection $\mathbb{P} f$ of the body force.
The presence of $\mathbb{P} f$ in 
\cref{eq:stokes_wf_reduced_helmholtz}
will turn out to play an important role in the pressure-robustness
of our error estimates in \Cref{sec:error_analysis}, 
and we discuss why this is the case in \Cref{rmk:pr_data_osc}.

\subsection{Mesh-related notation}
\label{subsec:mesh}

Let $\mathcal{T} = \{K\}$ be a conforming triangulation of $\Omega$ into
simplices $\{K\}$.  Let $K \in \mathcal{T}$. We use $\mathcal{F}_K$ to indicate
the collection of $(d-1)$-dimensional faces of $K$. We set
$h_K=\textrm{diam}(K)$ and let $n_K$ denote the outward unit normal on
$\partial K$.  The mesh size is defined as $h := \max_{K \in \mathcal{T}} h_K$
and the mesh skeleton is defined as
$\Gamma_0 = \bigcup_{K \in \mathcal{T}} \partial K$.

Notice that, if cracks are present in the domain, it is possible for two
distinct elements $K_1, K_2 \in \mathcal{T}$ to share a face $\sigma \in
\mathcal{F}_{K_1} \cap \mathcal{F}_{K_2}$ that lies on the boundary, i.e.
$\sigma \subset \partial \Omega$.  In this case, it will not be convenient to
view $\sigma$ as a single mesh face, as is typically done for interior faces.
Following \cite{veeser2017quasi_nonlipschitz}, we therefore define the
collection of mesh faces as the quotient set
\begin{align*}
	&\mathcal{F}_h :=
	\big\{(\sigma, K) : K \in \mathcal{T}, \ \sigma \in \mathcal{F}_K \big\}
	/{\sim}, \\
	(\sigma_1, K_1) \sim (\sigma_2, K_2) &\iff
	\big[ (\sigma_1, K_1) = (\sigma_2, K_2) \big] \textrm{ or }
	\big[\sigma_1 = \sigma_2 \textrm{ and }
	\sigma_1 \nsubseteq \partial \Omega \big].
\end{align*}
For $F = [(\sigma, K)] \in \mathcal{F}_h$ we set $h_F := \textrm{diam}(\sigma)$.
Also, surface integration on $F$ is well-defined, with the understanding that
$\int_F v \dif{s} := \int_\sigma v \dif{s}$ for all $v \in L^1(\sigma)$.  The
boundary faces $\mathcal{F}_b$ and interior faces $\mathcal{F}_i$ are naturally
defined as
\begin{equation*}
	\mathcal{F}_b :=
	\{ [(\sigma, K)] \in \mathcal{F}_h : \sigma \subset \partial \Omega \},
	\qquad \mathcal{F}_i := \mathcal{F}_h \setminus \mathcal{F}_b,
\end{equation*}
and we note that
$\partial \Omega = \bigcup_{[(\sigma, K)] \in \mathcal{F}_b} \sigma$
since the codimension of $\partial \Omega$ is one.

If $F \in \mathcal{F}_i$ is an interior face belonging to two distinct elements
$K_1, K_2 \in \mathcal{T}$, we let $n_F$ denote the unit normal on $F$ pointing
from $K_1$ to $K_2$, and we define on $F$ the jump operator $\jump{\cdot}$ and
average operator $\av{\cdot}$ in the usual way:
\begin{alignat}{2}
  \jump{\phi}|_F &:= \phi|_{K_1} - \phi|_{K_2},
  \label{eq:jump_def_interior} \\
  \av{\phi}|_F &:= \frac{1}{2}(\phi|_{K_1} + \phi|_{K_2}),
  \label{eq:av_def_interior}
\end{alignat}
where $\phi$ is any function defined piecewise on $K_1 \cup K_2$. The
ambiguity in the ordering of $K_1, K_2$ will be unimportant.  If $F \in
\mathcal{F}_b$ is a boundary face belonging to $K \in \mathcal{T}$, we
let $n_F$ denote the unit normal on $F$ outward to $K$, and we define on
$F$ the jump and average operators as
\begin{equation}
  \label{eq:jump_av_def_boundary}
  \jump{\phi}|_F = \av{\phi}|_F := \phi|_{K},
\end{equation}
where $\phi$ is any function defined on $K$.

Finally, the following definition will be used in \cref{sec:appendix_a}.
Let $K \in \mathcal{T}$.  Observe that we do not have $\mathcal{F}_K
\subset \mathcal{F}_h$ because of how $\mathcal{F}_h$ is defined using
equivalence classes.  We therefore define $\mathcal{F}_{K,h} := \{
[(\sigma, K)] \in \mathcal{F}_h : \sigma \in \mathcal{F}_K \}$ so that
$\mathcal{F}_{K,h} \subset \mathcal{F}_h$ holds. The sets
$\mathcal{F}_K$ and $\mathcal{F}_{K,h}$ intuitively encode the same
information; they both contain exactly $(d+1)$ elements
which describe the faces of $K$.  The only difference is that
$\mathcal{F}_{K,h}$ is defined using equivalence classes in
$\mathcal{F}_h$.

\subsection{Discrete finite element spaces and norms}
\label{subsec:discrete}

We introduce the following low-order discontinuous finite element
spaces on $\Omega$:
\begin{align*}
  X_h &:=
	\{v_h \in L^2(\Omega)^d :
	v_h|_K \in [\mathcal{P}_1(K)]^d \ \forall K \in \mathcal{T} \}, \\
  Q_h &:=
	\{q_h \in L_0^2(\Omega) :
	q_h|_K \in \mathcal{P}_0(K) \ \forall K \in \mathcal{T}\},
\end{align*}
where $\mathcal{P}_k(D)$ is the space of polynomials with degree at most
$k$ on $D$.  Also, let $\mathcal{P}_1(\mathcal{F}_h) := \prod_{F \in
\mathcal{F}_h} \mathcal{P}_1(F)$.
We introduce the low-order
discontinuous facet finite element spaces
\begin{align*}
  \bar{X}_h &:=
              \{\bar{v}_h \in [\mathcal{P}_1(\mathcal{F}_h)]^d :
              \bar{v}_h|_F = 0 \ \forall F \in
              \mathcal{F}_b \}, \\
  \bar{Q}_h &:= \mathcal{P}_1(\mathcal{F}_h).
\end{align*}
Notice that $\bar{X}_h$ can be viewed as the space of discontinuous
piecewise-linear vector functions on $\Gamma_0$ that vanish on $\partial
\Omega$.  Likewise, $\bar{Q}_h$ can be viewed as the space of
discontinuous piecewise-linear scalar functions on $\Gamma_0$, with the
caveat that these functions are double-valued on boundary faces shared
by two distinct cells.

It will also be convenient to introduce the extended velocity spaces
\begin{subequations} \label{eq:extended_velocity_spaces_def}
\begin{align} 
	X(h) &:= X_h + H_0^1(\Omega)^d,
	\\
	\bar{X}(h) &:= \bar{X}_h + H_0^{1/2}(\Gamma_0)^d,
\end{align}
\end{subequations}
where $H_0^{1/2}(\Gamma_0)^d$ is the trace space of functions in 
$H_0^1(\Omega)^d$ restricted to $\Gamma_0$.  
We use boldface notation for function pairs in 
$X(h) \times \bar{X}(h)$ and $Q_h \times \bar{Q}_h$,
i.e. 
\begin{equation*}
	\bm{v} = (v, \bar{v}) \in X(h) \times \bar{X}(h) 
	\quad \textrm{ and } \quad
	\bm{q}_h = (q_h, \bar{q}_h) \in Q_h \times \bar{Q}_h.
\end{equation*}

Throughout this paper 
$\nabla_h : X(h) \rightarrow [L^2(\Omega)]^{d \times d}$ 
denotes the broken gradient
$(\nabla_h v)|_K := \nabla (v|_K)$.  On the space $X(h)$ we introduce
the discrete $H^1$-norm
\begin{equation*}
  \dgnorm{v}^2 := \norm{\nabla_h v}^2 + \jsnorm{v}^2,
\end{equation*}
where $\jsnorm{\cdot}$ is the following jump semi-norm on $X(h)$:
\begin{equation*}
  \jsnorm{v}^2 := \sum_{F \in \mathcal{F}_h} \frac{1}{h_F} \norm{\jump{v}}_F^2.
\end{equation*}
Similarly, on the product space $X(h) \times \bar{X}(h)$ we introduce
the discrete $H^1$-norm
\begin{equation*}
  \vnorm{\bm{v}}^2 := \norm{\nabla_h v}^2 + \fsnorm{\bm{v}}^2,
\end{equation*}
where $\fsnorm{\cdot}$ is the following facet semi-norm on
$X(h) \times \bar{X}(h)$:
\begin{equation*}
  \fsnorm{\bm{v}}^2 := \sum_{K \in \mathcal{T}} \frac{1}{h_K} \norm{v - \bar{v}}_{\partial K}^2.
\end{equation*}
Finally, on the space $Q_h \times \bar{Q}_h$ we introduce the norm
\begin{equation*}
  \pnorm{\bm{q}_h}^2 := \norm{q_h}^2 + \pfnorm{\bar{q}_h}^2,
\end{equation*}
where $\pfnorm{\cdot}$ is the following norm on $\bar{Q}_h$:
\begin{equation*}
  \pfnorm{\bar{q}_h}^2 := \sum_{K \in \mathcal{T}} h_K \norm{\bar{q}_h}_{\partial K}^2.
\end{equation*}

We use $a \lesssim b$ to indicate $a \leq Cb$ where $C$ is a positive
constant depending only on $d, \Omega$ and the shape-regularity of
$\mathcal{T}$.  On occasion we will use inequalities of the form $a \leq C(s)
b$, where $C(s)$ is a positive constant depending only on $d, \Omega$,
shape-regularity of $\mathcal{T}$ and $s$, where $s \in [0, 1]$
corresponds to the order of the fractional Sobolev space
$H^{1+s}(\Omega)^d$.  In these cases, we will use the notation~$a
\lesssim_s b$.

We conclude this subsection with the observation that
\begin{equation}
  \label{eq:js_fs_semi-norm_ineq}
  \jsnorm{v} \lesssim \fsnorm{\bm{v}},
\end{equation}
which follows from the triangle inequality. Note that
\cref{eq:js_fs_semi-norm_ineq} implies
\begin{equation*}
  \dgnorm{v} \lesssim \vnorm{\bm{v}} .
\end{equation*}
These inequalities will be used frequently in \cref{sec:error_analysis}.

\subsection{The hybridized and embedded–hybridized discontinuous
  Galerkin methods}
\label{subsec:methods}

The lowest-order HDG and EDG--HDG methods analyzed in
\cite{rhebergen2020embedded} utilize the following finite element
spaces:
\begin{subequations}
	\begin{align}
		\bm{X}_h^v &:=
		\begin{cases}
			X_h \times \bar{X}_h & \quad \textrm{(HDG method)}, \\
			X_h \times (\bar{X}_h \cap C^0(\Gamma_0)^d) & \quad
			\textrm{(EDG--HDG method)},
		\end{cases} \label{eq:discrete_v_space_def} \\
		\bm{Q}_h^p &:= Q_h \times \bar{Q}_h.
	\end{align}
\end{subequations}
The HDG and HDG--EDG methods differ only in their choice of velocity
facet space, which is discontinuous for the HDG method and continuous
for the EDG--HDG method.  The remainder of the analysis is agnostic as to
whether the HDG or EDG--HDG method is considered, with the presented
analysis holding for both methods.

The discrete formulation of \cref{eq:stokes_wf} seeks
$(\bm{u}_h, \bm{p}_h) \in \bm{X}_h^v \times \bm{Q}_h^p$ such that
\begin{subequations}
  \label{eq:stokes_wf_discrete}
  \begin{alignat}{2}
    \nu a_h(\bm{u}_h, \bm{v}_h) + b_h(v_h, \bm{p}_h) &= (f, v_h)
    \qquad &&\fra \bm{v}_h \in \bm{X}_h^v, \\
    b_h(u_h, \bm{q}_h) &= 0 \qquad &&\fra \bm{q}_h \in \bm{Q}_h^p,
  \end{alignat}
\end{subequations}
where $a_h : \bm{X}_h^v \times \bm{X}_h^v \rightarrow \mathbb{R}$ and
$b_h : X_h \times \bm{Q}_h^p \rightarrow \mathbb{R}$ are the bilinear
forms
\begin{alignat}{2}
  \begin{split}
    \label{eq:a_h_def1}
    a_h(\bm{v}, \bm{w}) :=
    &\sum_{K \in \mathcal{T}}
    \int_K \nabla v : \nabla w \dif x
    + \sum_{K \in \mathcal{T}}
    \frac{\alpha}{h_K} \int_{\partial K}  (v - \bar{v}) \cdot (w - \bar{w}) \dif s
    \\
    &- \sum_{K \in \mathcal{T}}
    \int_{\partial K}
    \Big[
    (v - \bar{v}) \cdot \frac{\partial w}{\partial n_K}
    + (w - \bar{w}) \cdot \frac{\partial v}{\partial n_K}
    \Big] \dif s,
  \end{split}
  \\
  b_h(v, \bm{q})
  := &-\sum_{K \in \mathcal{T}} \int_K (\nabla \cdot v) q \dif x
  + \sum_{K \in \mathcal{T}} \int_{\partial K} (v \cdot n_K) \bar{q} \dif s,
  \label{eq:b_h_def}
\end{alignat}
and $\alpha > 0$ is a penalty parameter.  It was shown in
\cite[Lemma~4.2]{rhebergen2017analysis} that for sufficiently large
$\alpha$ the following coercivity result holds:
\begin{equation}
  \label{eq:a_h_coercive}
  \vnorm{\bm{v}_h}^2 \lesssim a_h(\bm{v}_h, \bm{v}_h) \qquad \fra \bm{v}_h \in
  \bm{X}_h^v.
\end{equation}
Let us also mention that inf-sup stability of $b_h$ was established in
\cite[Lemma~8]{rhebergen2020embedded}:
\begin{equation}
  \label{eq:b_h_infsup}
  \pnorm{\bm{q}_h} \lesssim
  \sup_{\bm{v}_h \in \bm{X}_h^v \backslash \{\bm{0}\}}
  \frac{b_h(v_h, \bm{q}_h)}{\vnorm{\bm{v}_h}}
  \qquad \forall \bm{q}_h \in \bm{Q}_h^p.
\end{equation}

A consequence of the stability properties in
\cref{eq:a_h_coercive,eq:b_h_infsup} is that the discrete problem
\cref{eq:stokes_wf_discrete} is well-posed, see
e.g.~\cite[Chapter~4]{boffi2013mixed}.  Furthermore, let us introduce
the discrete reduced space
\begin{align*}
  \bm{V}_h^v &:= \{\bm{v}_h \in \bm{X}_h^v : b_h(v_h, \bm{q}_h) = 0 \ \forall
  \bm{q}_h \in \bm{Q}_h^p \}
  \\
             &= \{\bm{v}_h \in \bm{X}_h^v : v_h \in X_h^{\mathrm{BDM}} \
             \textrm{and} \ \nabla \cdot v_h = 0 \},
\end{align*}
where $X_h^{\mathrm{BDM}}$ is the lowest-order Brezzi--Douglas--Marini (BDM)
space \cite{boffi2013mixed},
\begin{equation*}
  X_h^{\mathrm{BDM}}
  = \{v_h \in X_h : \jump{v_h}|_F \cdot n_F = 0
  \ \forall F \in \mathcal{F}_h \}.
\end{equation*}
Inf-sup stability of $b_h$ implies the best approximation result
\cite[Section 12.5]{brenner2008mathematical}
\begin{equation} \label{eq:reduced_vs_full_approx}
	\inf_{\tilde{\bm{v}}_h \in \bm{V}_h^v} \vnorm{\bm{u} - \tilde{\bm{v}}_h}
	\lesssim
	\inf_{\bm{v}_h \in \bm{X}_h^v} \vnorm{\bm{u} - \bm{v}_h},
\end{equation}
where $u \in H_0^1(\Omega)^d$ is the velocity solution to \cref{eq:stokes_wf}
and $\bm{u} = (u, u) \in X(h) \times \bar{X}(h)$.
Also, the discrete velocity solution $\bm{u}_h \in \bm{V}_h^v$ to
\cref{eq:stokes_wf_discrete} satisfies the discrete reduced problem
\begin{equation}
  \label{eq:stokes_wf_reduced_discrete}
  \nu a_h(\bm{u}_h, \bm{v}_h) = (f, v_h)
  \qquad \fra \bm{v}_h \in \bm{V}_h^v.
\end{equation}
However, for $\bm{v}_h \in \bm{V}_h^v$ it holds that $v_h \in
L_{\sigma}^2(\Omega)$ (recall that $L_{\sigma}^2(\Omega)$ is defined in
\cref{eqn:Lsigma}) and therefore $(f, v_h) = (\mathbb{P} f, v_h)$.
Hence the reduced problem \cref{eq:stokes_wf_reduced_discrete} can
equivalently be written as
\begin{equation}
  \label{eq:stokes_wf_reduced_discrete_helmholtz}
  \nu a_h(\bm{u}_h, \bm{v}_h) = (\mathbb{P} f, v_h)
  \qquad \fra \bm{v}_h \in \bm{V}_h^v.
\end{equation}
Analogously to \cref{eq:stokes_wf_reduced_helmholtz},
the presence of $\mathbb{P} f$ in 
\cref{eq:stokes_wf_reduced_discrete_helmholtz}
will play an important role in the pressure-robustness of our error estimates
in \Cref{sec:error_analysis}.

\subsection{Enrichment and interpolation operators}
\label{subsec:enrichment}

Our minimal regularity error analysis of the lowest-order HDG and
EDG--HDG methods will utilize an \emph{enrichment operator} $E_h$ with
the following properties.
\begin{lemma}[Enrichment operator]
  \label{lemma:enrichment}
  There exists a linear operator \\
  $E_h : X_h^{\mathrm{BDM}} \rightarrow H_0^1(\Omega)^d$ such that for all
  $v_h \in X_h^{\mathrm{BDM}}$ we have
  \begin{enumerate}[(i)]
    \setlength\itemsep{0.5em}
  \item $\int_F \av{v_h} \dif s = \int_F E_h v_h \dif s$ for all $F \in \mathcal{F}_i$.
  \label{itm:e_h_facet_mean}
  \item $\nabla \cdot v_h = \nabla \cdot E_h v_h$.
  \label{itm:e_h_div}
  \item $\sum_{K \in \mathcal{T}} h_K^{2(k-1)} \abs{v_h - E_h v_h}_{k,
  K}^2 \lesssim \jsnorm{v_h}^2$ for all $k \in \{0, 1\}$.
    \label{itm:e_h_interp}
  \item $\norm{\nabla E_h v_h} = \dgnorm{E_h v_h} \lesssim \dgnorm{v_h}$.
  \label{itm:e_h_bounded}
  \end{enumerate}
\end{lemma}
An operator satisfying the statements in \cref{lemma:enrichment} was
constructed in \cite{moment_div_preserving}. In
\cite{moment_div_preserving} the construction is outlined in detail for
the two-dimensional case, but sketched only briefly for the
three-dimensional case.  We present an alternative proof of
\cref{lemma:enrichment} for the three-dimensional case in
\cref{sec:appendix_a}.  Our construction is based on the conforming and
divergence-free finite element of \cite{guzman_neilan_3d_stokes} and is
inspired by \cite[Lemma~4.7]{linke2018quasi}, in which a similar result
is established for Crouzeix--Raviart finite element functions.  We
mention in passing that our construction in \cref{sec:appendix_a} can
also be adapted to the two-dimensional case by using the two-dimensional
finite element of~\cite{guzman2014conforming}.

Let $X_h^c := \{v_h \in X_h \cap C^0(\overline{\Omega})^d :
v_h|_{\partial \Omega} = 0 \}$ 
be the conforming analogue of $X_h$. Aside from $E_h$, we
will also use the following quasi-interpolation operator $I_h$ to deduce
optimal rates of convergence for the HDG and EDG--HDG methods.
\begin{lemma}[Quasi-interpolation operator]
	\label{lemma:interp_op}
	There exists a linear operator $I_h : H_0^1(\Omega)^d \rightarrow X_h^c$
	such that for all  $s \in [0, 1]$ and 
	$v \in H_0^1(\Omega)^d \cap H^{1+s}(\Omega)^d$
	we have
	\begin{equation} \label{eq:inter_op_estimate}
		\norm{\nabla_h (v - I_h v)} \lesssim_s h^s \norm{v}_{1+s}.
	\end{equation}
\end{lemma}
\begin{proof}
A proof of \cref{lemma:interp_op} can be found in
\cite{ern2017finitequasi,scott1990finite}, although these works assume
that $\Omega$ is a Lipschitz domain. For the sake of completeness, we
now show that the quasi-interpolation operator of
\cite{ern2017finitequasi} still satisfies \cref{eq:inter_op_estimate}
when it is not assumed that $\Omega$ is Lipschitz. The
quasi-interpolation operator of \cite{ern2017finitequasi} is given by
$I_h = A_h \circ \Pi_h$, where $A_h : X_h \rightarrow X_h^c$ is the
lowest-order averaging operator introduced in
\cite[Theorem~2.2]{karakashian2003posteriori}
and $\Pi_h : H_0^1(\Omega)^d \rightarrow X_h$ 
the $L^2$-orthogonal projector onto $X_h$. 
Because we are assuming that $\partial \Omega$ has codimension
one, every boundary vertex of the mesh is contained in some boundary
face of the mesh.  As a result, by the arguments used in
\cite[Theorem~2.2]{karakashian2003posteriori}:
\begin{equation} \label{eq:av_op_property}
	\norm{\nabla_h(v_h - A_h v_h)} \lesssim \jsnorm{v_h}.
\end{equation}
Let $v \in H_0^1(\Omega)^d \cap H^{1+s}(\Omega)^d$. 
Using the triangle inequality,
\cref{eq:av_op_property}, and a continuous trace inequality
\cite[Lemma~1.49]{di2011mathematical}, we have
\begin{align} \label{eq:interp_op_estimate1}
	\begin{split}
		\norm{\nabla_h (v - I_h v)} &\leq
		\norm{\nabla_h (v - \Pi_h v)} + \norm{\nabla_h (\Pi_h v - A_h \Pi_h v)}
		\\
		&\lesssim
		\norm{\nabla_h (v - \Pi_h v)} + \jsnorm{\Pi_h v} \\
		&= \norm{\nabla_h (v - \Pi_h v)} + \jsnorm{v - \Pi_h v} \\
		&\lesssim
		\Big( \sum_{K \in \mathcal{T}}
		h_K^{-2} \norm{v - \Pi_h v}_{K}^2
		+ \abs{v - \Pi_h v}_{1, K}^2 \Big)^{1/2}.
	\end{split}
\end{align}
Finally, \cref{eq:inter_op_estimate} follows from
\cref{eq:interp_op_estimate1}
and standard approximation properties of the $L^2$-orthogonal projector
$\Pi_h$ (see e.g.~\cite[Section~1.4.4]{di2011mathematical}). \qed
\end{proof}

\section{Pressure-robust error analysis under minimal regularity}
\label{sec:error_analysis}

In \cite{rhebergen2020embedded}, optimal and pressure-robust error
estimates for the HDG and EDG--HDG methods were derived assuming 
$u \in H_0^1(\Omega)^d \cap H^2(\Omega)^d$.  
In this section, we carry out error analysis
for the more general case of 
$u \in H_0^1(\Omega)^d \cap H^{1+s}(\Omega)^d$ 
for $s \in [0, 1]$.

\subsection{Velocity error estimates}
\label{sec:velocity_error}

Thus far we have considered $a_h$ on the finite element space $\bm{X}_h^v$ 
(see \cref{eq:discrete_v_space_def}).
The first step in our analysis is to extend $a_h$ to the larger space
$\bm{X}^v(h) := X(h) \times \bar{X}(h)$ 
(see \cref{eq:extended_velocity_spaces_def}). The main difficulty is that for 
$v\in X(h)$ and $K \in \mathcal{T}$ we have only $\nabla v \in [L^2(K)]^{d
\times d}$ and therefore $\nabla v$ does not admit a well-defined trace
on $\partial K$.  To deal with this problem, let $\pi_K : [L^2(K)]^{d
\times d} \rightarrow [\mathcal{P}_0(K)]^{d \times d}$ denote the
$L^2$-orthogonal projector onto $[\mathcal{P}_0(K)]^{d \times d}$. Hence
$(G - \pi_K G, H)_K = 0$ for all $G \in [L^2(K)]^{d \times d}$ and $H
\in [\mathcal{P}_0(K)]^{d \times d}$.  For any $\bm{v}, \bm{w} \in
\bm{X}^v(h)$ we now define
\begin{alignat}{2}
  \begin{split} \label{eq:a_h_def2}
    a_h(\bm{v}, \bm{w}) :=
    &\sum_{K \in \mathcal{T}}
    \int_K \nabla v : \nabla w \dif x
    + \sum_{K \in \mathcal{T}}
    \frac{\alpha}{h_K} \int_{\partial K}  (v - \bar{v}) \cdot (w - \bar{w}) \dif s \\
    &- \sum_{K \in \mathcal{T}}
    \int_{\partial K}
    \Big[
    (v - \bar{v}) \cdot ([\pi_K \nabla w] n_K)
    + (w - \bar{w}) \cdot ([\pi_K \nabla v] n_K)
    \Big] \dif s.
  \end{split}
\end{alignat}
We will use this bilinear form in the following analysis. Observe that
\cref{eq:a_h_def2} reduces to the previous definition of $a_h$ (see
\cref{eq:a_h_def1}) for $\bm{v}, \bm{w} \in \bm{X}_h^v$. Moreover, the
following boundedness result holds on the extended space $\bm{X}^v(h)$.

\begin{lemma}[Boundedness of $a_h$]
  \label{lemma:a_h_boundedness}
  For all $\bm{v}, \bm{w} \in \bm{X}^v(h)$ there holds
  \begin{equation*}
    a_h(\bm{v}, \bm{w}) \lesssim \vnorm{\bm{v}} \vnorm{\bm{w}}.
  \end{equation*}
\end{lemma}
\begin{proof}
  By definition we have that
  \begin{align*}
    a_h(\bm{v}, \bm{w}) =
    &\underbrace{\sum_{K \in \mathcal{T}}
      \int_K \nabla v : \nabla w \dif x}_{I_1}
      + \underbrace{\sum_{K \in \mathcal{T}}
      \frac{\alpha}{h_K} \int_{\partial K}  (v - \bar{v}) \cdot (w - \bar{w}) \dif s}_{I_2} \\
    &+ \underbrace{\sum_{K \in \mathcal{T}}
      -\int_{\partial K} (v - \bar{v}) \cdot ([\pi_K \nabla w] n_K) \dif s}_{I_3} \\
    &+ \underbrace{\sum_{K \in \mathcal{T}}
      -\int_{\partial K} (w - \bar{w}) \cdot ([\pi_K \nabla v] n_K) \dif s}_{I_4}.
  \end{align*}
  An application of the Cauchy--Schwarz inequality yields
  $\abs{I_1} + \abs{I_2} \lesssim \vnorm{\bm{v}} \vnorm{\bm{w}}$.  To
  bound $\abs{I_3}$ we first apply Cauchy--Schwarz to get
  \begin{align}
    \label{eq:a_h_boundedness_I_3}
    \begin{split}
      \abs{I_3} &\leq \fsnorm{\bm{v}}
      \Big(\sum_{K \in \mathcal{T}} h_K \norm{\pi_K \nabla w}_{\partial K}^2 \Big)^{1/2} \\
      &\lesssim \fsnorm{\bm{v}}
      \Big(\sum_{K \in \mathcal{T}}\norm{\pi_K \nabla w}_K^2 \Big)^{1/2} \\
      &\leq \fsnorm{\bm{v}}
      \Big(\sum_{K \in \mathcal{T}}\norm{\nabla w}_K^2 \Big)^{1/2} \\
      &\leq \vnorm{\bm{v}} \vnorm{\bm{w}}.
    \end{split}
  \end{align}
  For the second inequality in \cref{eq:a_h_boundedness_I_3} we used a
  discrete trace inequality, and for the third inequality we used
  stability of $\pi_K$.  Similar reasoning shows that
  $\abs{I_4} \lesssim \vnorm{\bm{v}} \vnorm{\bm{w}}$.  This completes the
  proof. \qed
\end{proof}

The next ingredient in our analysis is to establish an upper bound on
the consistency error for the velocity solution of the method in
\cref{eq:stokes_wf_discrete}.

\begin{lemma}[Consistency error for $a_h$]
  \label{lemma:a_h_consistency_error}
  Let $u \in H_0^1(\Omega)^d$ be the velocity solution of \cref{eq:stokes_wf}, 
  let $\bm{u} = (u, u)$, and let $\bm{u}_h \in \bm{X}_h^v$ be the discrete
  velocity solution of \cref{eq:stokes_wf_discrete}.  Then for
  all $\bm{v}_h \in \bm{X}_h^v$ and $\bm{w}_h \in \bm{V}_h^v$ it holds that
  \begin{align}
    \label{eq:a_h_consistency_error}
    a_h(\bm{u} - \bm{u}_h, \bm{w}_h) \lesssim
    \Big\{ \vnorm{\bm{u} - \bm{v}_h} + \gsnorm{v_h}  + \frac{1}{\nu} \osc{\mathbb{P}f} \Big\} \fsnorm{\bm{w}_h},
  \end{align}
  where we have introduced the notation
  \begin{alignat}{2}
    \osc{g}^2 &:= \sum_{K \in \mathcal{T}} h_K^2 \norm{g}_K^2
    \qquad &&\fra g \in L^2(\Omega)^d, \label{eq:osc_defn} \\
    \gsnorm{t_h}^2 &:= \sum_{F \in \mathcal{F}_i} h_F \norm{\jump{\nabla_h t_h} n_F}_F^2
    \qquad &&\fra t_h \in X_h.
  \end{alignat}
\end{lemma}
\begin{proof}
  Let $\bm{v}_h \in \bm{X}_h^v$ and $\bm{w}_h \in \bm{V}_h^v$. We set
  \begin{equation*}
    \bm{z}_h = \bm{w}_h - (E_h w_h, E_h w_h) = (w_h - E_h w_h, \bar{w}_h - E_h w_h).
  \end{equation*}
  Then
  \begin{align}
    \label{eq:a_h_consistency_three_terms}
    \begin{split}
      a_h(\bm{u} - \bm{u}_h, \bm{w}_h)
      =& \big[a_h(\bm{u}, (E_h w_h, E_h w_h)) - a_h(\bm{u}_h, \bm{w}_h)\big] \\
      &+ a_h(\bm{u} - \bm{v}_h, \bm{z}_h)
      + a_h(\bm{v}_h, \bm{z}_h) \\
      =& \underbrace{\big[a(u, E_h w_h) - a_h(\bm{u}_h, \bm{w}_h)\big]}_{I_1}
      + \underbrace{a_h(\bm{u} - \bm{v}_h, \bm{z}_h)}_{I_2}
      + \underbrace{a_h(\bm{v}_h, \bm{z}_h)}_{I_3}.
    \end{split}
  \end{align}
  We first bound $I_1$. Since $\bm{w}_h \in \bm{V}_h^v$ we have
  $w_h \in X_h^{\mathrm{BDM}}$ with $\nabla \cdot w_h = 0$.  Thus
  $E_h w_h \in V$ by \cref{itm:e_h_div} of \cref{lemma:enrichment}.
  Using the reduced problems \cref{eq:stokes_wf_reduced_helmholtz} and
  \cref{eq:stokes_wf_reduced_discrete_helmholtz}, the Cauchy--Schwarz
  inequality, and \cref{itm:e_h_interp} of \cref{lemma:enrichment} with
  $k=0$,
  \begin{align}
    \label{eq:a_h_consistency_I_1_bound}
    \begin{split}
      I_1 &= \frac{1}{\nu} (\mathbb{P} f, E_h w_h - w_h) \\
      &\leq \frac{1}{\nu} \osc{\mathbb{P} f} \Bigg[ \sum_{K \in \mathcal{T}} h_K^{-2} \norm{E_h w_h - w_h}_K^2 \Bigg]^{1/2} \\
      &\lesssim \frac{1}{\nu} \osc{\mathbb{P} f} \jsnorm{w_h} \\
      &\lesssim \frac{1}{\nu} \osc{\mathbb{P} f} \fsnorm{\bm{w}_h}.
    \end{split}
  \end{align}
  We now bound $I_2$. Using \cref{itm:e_h_interp} of
  \cref{lemma:enrichment} with $k=1$ we have
  \begin{equation*}
    \vnorm{\bm{z}_h}^2
    = \norm{\nabla_h (w_h - E_h w_h)}^2 + \fsnorm{\bm{w}_h}^2
    \lesssim \fsnorm{\bm{w}_h}^2
  \end{equation*}
  so that $\vnorm{\bm{z}_h} \lesssim \fsnorm{\bm{w}_h}$.  Hence by
  \cref{lemma:a_h_boundedness} we have
  \begin{equation}
    \label{eq:a_h_consistency_I_2_bound}
    I_2 \lesssim \vnorm{\bm{u} - \bm{v}_h} \vnorm{\bm{z}_h} \lesssim \vnorm{\bm{u} - \bm{v}_h} \fsnorm{\bm{w}_h}.
  \end{equation}
  To bound $I_3$ we use the definition \cref{eq:a_h_def2} of $a_h$ and
  integrate by parts element-wise.  Using that $(\nabla^2 v_h)|_K = 0$
  as $v_h$ is piecewise linear, this results in
  \begin{align}
    \label{eq:a_h_consistency_I_3_bound_1}
    \begin{split}
      I_3 = &\underbrace{\sum_{K \in \mathcal{T}}
	\frac{\alpha}{h_K} \int_{\partial K}  (v_h - \bar{v}_h) \cdot (z_h - \bar{z}_h) \dif s
	- \int_{\partial K} (v_h - \bar{v}_h) \cdot ([\pi_K \nabla z_h] n_K) \dif s}_{I_{3, 1}} \\
      &+ \underbrace{\sum_{K \in \mathcal{T}}
	\int_{\partial K} \bar{z}_h \cdot ([\nabla v_h] n_K) \dif s}_{I_{3, 2}}.
    \end{split}
  \end{align}
  Using the same arguments from \cref{lemma:a_h_boundedness} (namely
  those used in \cref{eq:a_h_boundedness_I_3}) one sees that
  \begin{equation}
    \label{eq:a_h_consistency_I_31}
    I_{3, 1} \lesssim \fsnorm{\bm{v}_h} \vnorm{\bm{z}_h}
    = \fsnorm{\bm{u} - \bm{v}_h} \vnorm{\bm{z}_h}
    \lesssim \vnorm{\bm{u} - \bm{v}_h} \fsnorm{\bm{w}_h}.
  \end{equation}
  Also, rewriting $I_{3, 2}$ in terms of facet integrals, applying
  \cref{itm:e_h_facet_mean} of \cref{lemma:enrichment}, and using the
  Cauchy--Schwarz inequality, we find
  \begin{align}
    \label{eq:a_h_consistency_I_32}
    \begin{split}
      I_{3, 2} &= \sum_{F \in \mathcal{F}_i} \int_F \bar{z}_h \cdot (\jump{\nabla v_h} n_F) \dif s \\
      &= \sum_{F \in \mathcal{F}_i} \int_F (\bar{w}_h - E_h w_h) \cdot (\jump{\nabla v_h} n_F) \dif s \\
      &= \sum_{F \in \mathcal{F}_i} \int_F (\bar{w}_h - \av{w_h}) \cdot (\jump{\nabla v_h} n_F) \dif s \\
      &\lesssim \fsnorm{\bm{w}_h} \gsnorm{v_h}.
    \end{split}
  \end{align}
  Using \cref{eq:a_h_consistency_I_31} and
  \cref{eq:a_h_consistency_I_32} in
  \cref{eq:a_h_consistency_I_3_bound_1} we obtain
  \begin{equation}
    \label{eq:a_h_consistency_I_3_bound_2}
    I_3 \lesssim \Big[\vnorm{\bm{u} - \bm{v}_h} + \gsnorm{v_h} \Big] \fsnorm{\bm{w}_h}.
  \end{equation}
  Finally, using the bounds
  \cref{eq:a_h_consistency_I_1_bound,eq:a_h_consistency_I_2_bound,eq:a_h_consistency_I_3_bound_2}
  in \cref{eq:a_h_consistency_three_terms} yields the desired result
  \cref{eq:a_h_consistency_error}.  \qed
\end{proof}

With boundedness and consistency results established for the method
\cref{eq:stokes_wf_discrete}, we can now derive our main error
estimate.

\begin{theorem}[Velocity error]
  \label{thm:main_u_estimate}
  Let $u \in H_0^1(\Omega)^d$ be the velocity solution of \cref{eq:stokes_wf}, 
  let $\bm{u} = (u, u)$, and let $\bm{u}_h \in \bm{X}_h^v$ be the discrete
  velocity solution of \cref{eq:stokes_wf_discrete}. Then
  \begin{equation}
    \label{eq:main_u_estimate}
    \vnorm{\bm{u} - \bm{u}_h} \lesssim
    \inf_{\bm{v}_h \in \bm{X}_h^v}
    \Big[
    \vnorm{\bm{u} - \bm{v}_h}
    + \gsnorm{v_h}
    \Big]
    + \frac{1}{\nu} \osc{\mathbb{P}f}.
  \end{equation}
\end{theorem}
\begin{proof}
  Let $\bm{v}_h \in \bm{X}_h^v$.  Owing to \cref{eq:reduced_vs_full_approx}
  we can find $\tilde{\bm{v}}_h \in \bm{V}_h^v$ with
  $\vnorm{\bm{u} - \tilde{\bm{v}}_h} \lesssim \vnorm{\bm{u} - \bm{v}_h}$.
  Let $\bm{w}_h = (\bm{u}_h - \tilde{\bm{v}}_h) \in \bm{V}_h^v$.  Using discrete
  coercivity \cref{eq:a_h_coercive} along with the boundedness and
  consistency results
  \crefrange{lemma:a_h_boundedness}{lemma:a_h_consistency_error},
  \begin{align}
    \label{eq:main_u_estimate_part1}
    \begin{split}
      \vnorm{\bm{w}_h}^2 &\lesssim a_h(\bm{w}_h, \bm{w}_h) \\
      &= a_h(\bm{u}_h - \bm{u}, \bm{w}_h) + a_h(\bm{u} - \tilde{\bm{v}}_h,
      \bm{w}_h) \\
      &\lesssim
      \Big\{
      \Big[
      \vnorm{\bm{u} - \tilde{\bm{v}}_h}
      + \gsnorm{\tilde{v}_h}
      \Big]
      + \frac{1}{\nu} \osc{\mathbb{P}f}
      \Big\} \vnorm{\bm{w}_h}.
    \end{split}
  \end{align}
  Therefore, dividing \cref{eq:main_u_estimate_part1} by
  $\vnorm{\bm{w}_h}$ we arrive at
  \begin{equation}
    \label{eq:main_u_estimate_part2}
    \vnorm{\bm{u}_h - \tilde{\bm{v}}_h} \lesssim
    \Big[
    \vnorm{\bm{u} - \tilde{\bm{v}}_h}
    + \gsnorm{\tilde{v}_h}
    \Big]
    + \frac{1}{\nu} \osc{\mathbb{P}f}.
  \end{equation}
  Using the triangle inequality and
  \cref{eq:main_u_estimate_part2} we obtain
  \begin{align}
  \begin{split}
  	\label{eq:main_u_estimate_part3}
    \vnorm{\bm{u} - \bm{u}_h} &\leq \vnorm{\bm{u} - \tilde{\bm{v}}_h} +
    \vnorm{\bm{u}_h
    - \tilde{\bm{v}}_h}
    \\
    &\lesssim
    \Big[
    \vnorm{\bm{u} - \tilde{\bm{v}}_h}
    + \gsnorm{\tilde{v}_h}
    \Big]
    + \frac{1}{\nu} \osc{\mathbb{P}f}.
  \end{split}
  \end{align}
  Also, by the triangle inequality and a discrete trace inequality we have
  \begin{align}
  \begin{split}
  	\label{eq:main_u_estimate_part4}
    \gsnorm{\tilde{v}_h} &\leq \gsnorm{\tilde{v}_h - v_h} + \gsnorm{v_h}
    \\
    &\lesssim \norm{\nabla_h (\tilde{v}_h - v_h)} + \gsnorm{v_h}
    \\
    &\leq \vnorm{\tilde{\bm{v}}_h - \bm{v}_h} + \gsnorm{v_h}
    \\
    &\leq \vnorm{\bm{u} - \tilde{\bm{v}}_h} + \vnorm{\bm{u} - \bm{v}_h} + \gsnorm{v_h}.
  \end{split}
  \end{align}
  Combining \crefrange{eq:main_u_estimate_part3}{eq:main_u_estimate_part4}
  and using that $\vnorm{\bm{u} - \tilde{\bm{v}}_h} \lesssim \vnorm{\bm{u} -
  \bm{v}_h}$
  we obtain
  \begin{equation*}
  	\vnorm{\bm{u} - \bm{u}_h} \lesssim
  	\Big[
  	\vnorm{\bm{u} - \bm{v}_h}
  	+ \gsnorm{v_h}
  	\Big]
  	+ \frac{1}{\nu} \osc{\mathbb{P}f}.
  \end{equation*}
  The desired result \cref{eq:main_u_estimate} follows as
  $\bm{v}_h \in \bm{X}_h^v$ is arbitrary.  \qed
\end{proof}

\begin{remark}[Pressure-robustness of the data oscillation term]
\label{rmk:pr_data_osc}
  As discussed in \cite[Remark 5.9]{linke2018quasi}, the function
  $\frac{1}{\nu}\mathbb{P}f$ is independent of both the
  pressure $p$ and the viscosity $\nu$.
  This can be seen by
  extending the domain of the Helmholtz projector $\mathbb{P}$ to
  $[H^{-1}(\Omega)]^d$, and then utilizing the fact that
  $-\nu \Updelta u + \nabla p = f$ holds in the distributional sense.
  One finds that
  \begin{equation}
    \label{eq:data_osc_pr}
    \frac{1}{\nu}\mathbb{P}f
    = \frac{1}{\nu}\mathbb{P}(-\nu \Updelta u + \nabla p)
    = \mathbb{P}(-\Updelta u) \in L^2(\Omega)^d,
  \end{equation}
  since $\mathbb{P}(\nabla p) = 0$.  We refer the reader to
  \cite[Section 3]{linke2019pressure} for a more detailed discussion
  of these ideas.  A consequence of \cref{eq:data_osc_pr} is that the
  data oscillation term appearing in \cref{eq:main_u_estimate}
  can equivalently be written as
  $\frac{1}{\nu} \osc{\mathbb{P}f} = \osc{\mathbb{P}(-\Updelta u)}$.
  Because this quantity depends only on the velocity,
  the error estimate \cref{eq:main_u_estimate} is pressure-robust.
  We emphasize that 
  \cref{eq:main_u_estimate}
  would not be a pressure-robust error estimate
  if it contained $\frac{1}{\nu} \osc{f}$
  instead of $\frac{1}{\nu} \osc{\mathbb{P} f}$.
\end{remark}

Our next step is to show that the interpolation error term
appearing in \cref{thm:main_u_estimate} converges optimally
with respect to the mesh size $h$.

\begin{lemma}[Interpolation error]
  \label{lemma:interp_error}
  Let $\psi \in H_0^1(\Omega)^d \cap H^{1+s}(\Omega)^d$ 
  with $s \in [0, 1]$, and set $\bm{\psi} = (\psi, \psi)$. Then
  \begin{align}
    \label{eq:interp_error_estimate}
    \inf_{\bm{v}_h \in \bm{X}_h^v}
    \Big[
    \vnorm{\bm{\psi} - \bm{v}_h}
    + \gsnorm{v_h}
    \Big]
    \lesssim_s h^s \norm{\psi}_{1+s}.
  \end{align}
\end{lemma}
\begin{proof}
  For each $F \in \mathcal{F}_i$ we introduce the patch $\omega_F$
  of elements sharing $F$,
  \begin{equation*}
  	\omega_F :=
  	\bigcup \{K \in \mathcal{T} : F \textrm{ is a face of } K\}.
  \end{equation*}
  Note that $\omega_F$ is the union of exactly two elements.
  Now let $\bm{v}_h \in \bm{X}_h^v$ and
  $g_F \in [\mathcal{P}_0(\omega_F)]^{d \times d}$
  be arbitrary.  By a discrete trace inequality and the triangle inequality,
  \begin{align}
    \label{eq:interp_error_gsnorm}
    \begin{split}
      \gsnorm{v_h}^2
      &\leq \sum_{F \in \mathcal{F}_i}
      h_F \norm{\jump{\nabla_h v_h}}_F^2 \\
      &= \sum_{F \in \mathcal{F}_i}
      h_F \norm{\jump{\nabla_h v_h - g_F}}_F^2 \\
      &\lesssim \sum_{F \in \mathcal{F}_i}
      \norm{\nabla_h v_h - g_F}_{\omega_F}^2 \\
      &\lesssim \sum_{F \in \mathcal{F}_i} \Big[
      \norm{\nabla_h (\psi - v_h)}_{\omega_F}^2
      + \norm{\nabla \psi - g_F}_{\omega_F}^2 \Big] \\
      &\lesssim \vnorm{\bm{\psi} - \bm{v}_h}^2
      + \sum_{F \in \mathcal{F}_i} \norm{\nabla \psi - g_F}_{\omega_F}^2.
    \end{split}
  \end{align}
  Since $\bm{v}_h \in \bm{X}_h^v$ and
  $g_F \in [\mathcal{P}_0(\omega_F)]^{d \times d}$
  are arbitrary, \cref{eq:interp_error_gsnorm} yields that
  \begin{align}
  \begin{split}
    \label{eq:interp_error_two_terms}
    \inf_{\bm{v}_h \in \bm{X}_h^v}
    \Big[
    \vnorm{\bm{\psi} - \bm{v}_h}
    + \gsnorm{v_h}
    \Big]
    \lesssim
    &\underbrace{\Big(
   	\sum_{F \in \mathcal{F}_i}
   	\inf_{g_F \in [\mathcal{P}_0(\omega_F)]^{d \times d}}
   	\norm{\nabla \psi - g_F}_{\omega_F}^2
   	\Big)^{1/2}}_{I_1}
     \\
    &+ \underbrace{\inf_{\bm{v}_h \in \bm{X}_h^v} \vnorm{\bm{\psi} - \bm{v}_h}}_{I_2}.
  \end{split}
  \end{align}
  A bound for $I_1$ follows from the fractional order Bramble--Hilbert lemma
  \cite[Theorem~6.1]{dupont1980polynomial} applied to the patches $\omega_F$:
  \begin{equation}
  \label{eq:interp_error_I1}
  I_1 \lesssim_s
  \Big(\sum_{F \in \mathcal{F}_i}
  h_F^{2s} \norm{\nabla \psi}_{s, \omega_F}^2 \Big)^{1/2}
  \lesssim h^s \norm{\psi}_{1+s}.
  \end{equation}
  To bound $I_2$ we take $\bm{v}_h = (I_h \psi, I_h \psi) \in \bm{X}_h^v$
  where $I_h$ is the quasi-interpolation operator introduced in
  \cref{lemma:interp_op}.  We find that
  \begin{equation}
  \label{eq:interp_error_I2}
  I_2 \leq \vnorm{\bm{\psi} - \bm{v}_h}
  = \norm{\nabla_h(\psi - I_h \psi)}
  \lesssim_s h^s \norm{\psi}_{1+s}.
  \end{equation}
  Using the bounds \crefrange{eq:interp_error_I1}{eq:interp_error_I2}
  in \cref{eq:interp_error_two_terms} yields the desired result. \qed
\end{proof}

An immediate consequence of \cref{thm:main_u_estimate} and
\cref{lemma:interp_error} is the following error estimate, which is
pressure-robust and optimal in the discrete energy norm.

\begin{corollary}[Pressure-robust error estimate]
  \label{coro:conv_result}
  In addition to the assumptions of \cref{thm:main_u_estimate}, assume
  that $u \in H^{1+s}(\Omega)^d$ with $s \in [0, 1]$.  Then
  \begin{equation*}
    \vnorm{\bm{u} - \bm{u}_h}
    \lesssim_s
    h^s \norm{u}_{1+s}
    + \frac{1}{\nu} \osc{\mathbb{P}f}.
  \end{equation*}
  Also, owing to \cref{eq:osc_defn}, the data oscillation term can be
  estimated as
  \begin{equation*}
  	\frac{1}{\nu} \osc{\mathbb{P}f} \leq h \norm{\frac{1}{\nu} \mathbb{P}f}.
  \end{equation*}
\end{corollary}

\begin{remark}[Convergence under $H^1$-regularity]
  In the case of $s=0$, where only $H^1$-regularity of $u$ is assumed,
  \cref{coro:conv_result} does not predict that
  $\vnorm{\bm{u} - \bm{u}_h} \rightarrow 0$ as $h \rightarrow 0$.
  This can still be proven, however, using \cref{thm:main_u_estimate}
  and a density argument.  Indeed, let $\epsilon > 0$.
  By definition, $H_0^1(\Omega)^d$ is the closure of $C_0^{\infty}(\Omega)^d$ 
  under the $H^1$-norm. As $u \in H_0^1(\Omega)^d$ we can therefore find
  $\phi \in C_0^{\infty}(\Omega)^d$ with
  $\abs{u - \phi}_1 < \epsilon$.  Setting $\bm{\phi} = (\phi, \phi)$,
  the triangle inequality and \cref{lemma:interp_error} then yield
  \begin{align}
    \begin{split} \label{eq:lowest_reg_limit_1}
      \inf_{\bm{v}_h \in \bm{X}_h^v}
      \Big[
      \vnorm{\bm{u} - \bm{v}_h}
      + \gsnorm{v_h}
      \Big]
      &\leq
      \vnorm{\bm{u} - \bm{\phi}} +
      \inf_{\bm{v}_h \in \bm{X}_h^v}
      \Big[
      \vnorm{\bm{\phi} - \bm{v}_h}
      + \gsnorm{v_h}
      \Big] \\
      &= \abs{u - \phi}_1 +
      \inf_{\bm{v}_h \in \bm{X}_h^v}
      \Big[
      \vnorm{\bm{\phi} - \bm{v}_h}
      + \gsnorm{v_h}
      \Big] \\
      &\lesssim \epsilon + h \norm{\phi}_{2} \\
      &\leq 2 \epsilon,
    \end{split}
  \end{align}
  where the last inequality in \cref{eq:lowest_reg_limit_1} holds for
  $h$ sufficiently small.  But $\epsilon > 0$ is arbitrary, and
  therefore \cref{eq:lowest_reg_limit_1} implies that
  \begin{equation}
    \label{eq:lowest_reg_limit_2}
    \lim_{h \rightarrow 0} \Big\{
    \inf_{\bm{v}_h \in \bm{X}_h^v}
    \Big[
    \vnorm{\bm{u} - \bm{v}_h}
    + \gsnorm{v_h}
    \Big] \Big\}= 0.
  \end{equation}
  Using \cref{eq:lowest_reg_limit_2} in \cref{thm:main_u_estimate} we
  see that $\vnorm{\bm{u} - \bm{u}_h} \rightarrow 0$ as
  $h \rightarrow 0$.
\end{remark}

We now investigate convergence of the velocity in the $L^2$-norm, by
means of the Aubin--Nitsche trick.  In order to proceed we assume the
domain $\Omega$ is such that the following regularity holds
(see e.g.~\cite{dauge_stokesreg2}).

\begin{assumption}[Regularity of the reduced Stokes problem]
  \label{asm:dual_reg}
  Let $s_0 \in [0, 1]$ be fixed.  We assume that for all
  $g \in L^2(\Omega)^d$ there holds
  $\phi_g \in H^{1+s_0}(\Omega)^d$ and
  \begin{align*}
    \norm{\phi_g}_{1+s_0} \lesssim_{s_0} \norm{g}
  \end{align*}
  where $\phi_g \in V$ is the solution to the reduced Stokes problem
  \begin{align*}
    a(\phi_g, v) = (g, v) \qquad \fra v \in V.
  \end{align*}
\end{assumption}

\begin{theorem}[Velocity error in the $L^2$-norm]
  \label{thm:v_conv_l2}
  In addition to the assumptions of \cref{coro:conv_result} and under
  \cref{asm:dual_reg} we have
  \begin{subequations}
  \begin{align}
  	\norm{u - u_h}
  	&\lesssim_{s_0} h^{s_0}
  	\Big\{
  	\inf_{\bm{v}_h \in \bm{X}_h^v}
  	\Big[
  	\vnorm{\bm{u} - \bm{v}_h}
  	+ \gsnorm{v_h}
  	\Big]
  	+ \frac{1}{\nu} \osc{\mathbb{P}f} \Big\} \label{eq:l2_est_1} \\
  	&\lesssim_s h^{s + s_0} \norm{u}_{1+s}
  	+ h^{1 + s_0} \norm{\frac{1}{\nu} \mathbb{P}f}. \label{eq:l2_est_2}
  \end{align}
  \end{subequations}
\end{theorem}
\begin{proof}
  Let $\phi \in V, \bm{\phi}_h \in \bm{V}_h^v$ solve the reduced
  problems
  \begin{subequations}
    \label{eq:phi_problems}
    \begin{alignat}{2}
      a(\phi, v) &= (u - u_h, v)  \qquad &&\fra v \in V,
      \label{eq:phi_cts} \\
      a_h(\bm{\phi}_h, \bm{v}_h) &= (u - u_h, v_h) \qquad &&\fra \bm{v}_h \in \bm{V}_h^v.
      \label{eq:phi_discrete}
    \end{alignat}
  \end{subequations}
  Set $\bm{\phi} = (\phi, \phi)$.  By \cref{asm:dual_reg} we have
  $\phi \in H^{1+{s_0}}(\Omega)^d$ and
  $\norm{\phi}_{1+s_0} \lesssim_{s_0} \norm{u - u_h}$.  Applying
  \cref{coro:conv_result} to the reduced problems
  \cref{eq:phi_problems} (for which the source term is $u - u_h$ and
  the viscosity is one), we find that
  \begin{align}
    \label{eq:phi_v_err_bound}
    \begin{split}
      \vnorm{\bm{\phi} - \bm{\phi}_h} &\lesssim_{s_0}
      h^{s_0} \norm{\phi}_{1+s_0} + \osc{\mathbb{P}(u - u_h)} \\
      &\leq h^{s_0}  \norm{\phi}_{1+s_0} + h \norm{u - u_h} \\
      &\lesssim_{s_0} h^{s_0} \norm{u - u_h}.
    \end{split}
  \end{align}

  Using \crefrange{eq:phi_cts}{eq:phi_discrete} and some algebraic
  manipulations, we have
  \begin{align}
    \label{eq:l2_three_terms}
    \begin{split}
      \norm{u - u_h}^2
      &= (u - u_h, u) - (u - u_h, u_h) \\
      &= a(\phi, u) - a_h(\bm{\phi_h}, \bm{u}_h) \\
      &= a_h(\bm{\phi}, \bm{u}) - a_h(\bm{\phi_h}, \bm{u}_h) \\
      &= \underbrace{a_h(\bm{u} - \bm{u}_h, \bm{\phi} - \bm{\phi}_h)}_{I_1}
      + \underbrace{a_h(\bm{\phi} - \bm{\phi}_h, \bm{u}_h)}_{I_2}
      + \underbrace{a_h(\bm{u} - \bm{u}_h, \bm{\phi}_h)}_{I_3}.
    \end{split}
  \end{align}
  To bound $I_1$ we use \cref{lemma:a_h_boundedness} and
  \cref{eq:phi_v_err_bound}:
  \begin{equation}
    \label{eq:l2_I1_bound}
    I_1 \lesssim \vnorm{\bm{u} - \bm{u}_h} \vnorm{\bm{\phi} - \bm{\phi}_h}
    \lesssim_{s_0} h^{s_0} \norm{u - u_h} \vnorm{\bm{u} - \bm{u}_h}.
  \end{equation}
  To bound $I_2$ we use \cref{lemma:a_h_consistency_error},
  \cref{lemma:interp_error} and \cref{asm:dual_reg}.  This yields
  \begin{align}
    \label{eq:l2_I2_bound}
    \begin{split}
      I_2 &\lesssim
      \Big\{
      \inf_{\bm{v}_h \in \bm{X}_h^v}
      \Big[
      \vnorm{\bm{\phi} - \bm{v}_h}
      + \gsnorm{v_h}
      \Big]
      + \osc{\mathbb{P}(u - u_h)}
      \Big\} \fsnorm{\bm{u}_h} \\
      &\lesssim_{s_0}
      \Big\{
      h^{s_0} \norm{\phi}_{1+s_0}
      + \osc{\mathbb{P}(u - u_h)}
      \Big\} \fsnorm{\bm{u}_h} \\
      &\lesssim_{s_0} h^{s_0} \norm{u - u_h} \fsnorm{\bm{u}_h} \\
      &\leq h^{s_0} \norm{u - u_h} \vnorm{\bm{u} - \bm{u}_h}.
    \end{split}
  \end{align}
  To bound $I_3$ we again use \cref{lemma:a_h_consistency_error} along
  with \cref{eq:phi_v_err_bound}. We find
  \begin{align}
    \label{eq:l2_I3_bound}
    \begin{split}
      I_3 &\lesssim
      \Big\{
      \inf_{\bm{v}_h \in \bm{X}_h^v}
      \Big[
      \vnorm{\bm{u} - \bm{v}_h}
      + \gsnorm{v_h}
      \Big]
      + \frac{1}{\nu} \osc{\mathbb{P}f}
      \Big\} \fsnorm{\bm{\phi}_h} \\
      &\leq
      \Big\{
      \inf_{\bm{v}_h \in \bm{X}_h^v}
      \Big[
      \vnorm{\bm{u} - \bm{v}_h}
      + \gsnorm{v_h}
      \Big]
      + \frac{1}{\nu} \osc{\mathbb{P}f}
      \Big\} \vnorm{\bm{\phi} - \bm{\phi}_h} \\
      &\lesssim_{s_0}
      h^{s_0} \norm{u - u_h}
      \Big\{
      \inf_{\bm{v}_h \in \bm{X}_h^v}
      \Big[
      \vnorm{\bm{u} - \bm{v}_h}
      + \gsnorm{v_h}
      \Big]
      + \frac{1}{\nu} \osc{\mathbb{P}f}
      \Big\}.
    \end{split}
  \end{align}
  Using the bounds \crefrange{eq:l2_I1_bound}{eq:l2_I3_bound} in
  \cref{eq:l2_three_terms}, and using \cref{thm:main_u_estimate} to
  bound $\vnorm{\bm{u} - \bm{u}_h}$, we obtain
  \begin{equation}
    \label{eq:l2_three_terms_done}
    \norm{u - u_h}^2
    \lesssim_{s_0}
    h^{s_0} \norm{u - u_h}
    \Big\{
    \inf_{\bm{v}_h \in \bm{X}_h^v}
    \Big[
    \vnorm{\bm{u} - \bm{v}_h}
    + \gsnorm{v_h}
    \Big]
    + \frac{1}{\nu} \osc{\mathbb{P}f}
    \Big\}.
  \end{equation}
  Dividing \cref{eq:l2_three_terms_done} by $\norm{u - u_h}$ we obtain
  \cref{eq:l2_est_1}.  Finally, \cref{eq:l2_est_2} follows from
  \cref{eq:l2_est_1} and \cref{lemma:interp_error}. \qed
\end{proof}

\subsection{Pressure error estimate}
\label{sec:pressure_error}

Let $(u, p) \in H_0^1(\Omega)^d \times L_0^2(\Omega)$ 
be the solution of \cref{eq:stokes_wf} and
$(\bm{u}_h, \bm{p}_h) \in \bm{X}_h^v \times \bm{Q}_h^p$ the solution of
\cref{eq:stokes_wf_discrete}.  Let $\pi_h : L^2(\Omega) \rightarrow Q_h$
be the $L^2$-orthogonal projector onto $Q_h$.  Note that $(\pi_h p -
p_h) \in Q_h$ and therefore $(p - \pi_h p, \pi_h p - p_h) = 0$.  Hence
by the Pythagorean theorem, the pressure error can be decomposed as
\begin{align} \label{eq:pressure_error_decomp}
\begin{split}
	\norm{p - p_h}^2
	&= \norm{(p - \pi_h p) + (\pi_h p -  p_h)}^2
	\\
	&= \norm{p - \pi_h p}^2 + \norm{\pi_h p -  p_h}^2.
\end{split}
\end{align}
The term $\norm{p - \pi_h p}$ is the best approximation error of $p$
by functions in the discrete space $Q_h$ under the $L^2$-norm, and is
unavoidably pressure-dependent.  However, the following theorem shows
that the second term $\norm{\pi_h p - p_h}$ can be bounded above by an
error that is dependent on the velocity only.

\begin{theorem}[Pressure error]
  \label{thm:pressure_error}
  Let $(u, p) \in H_0^1(\Omega)^d \times L_0^2(\Omega)$ 
  solve \cref{eq:stokes_wf} and set
  $\bm{u} = (u, u)$.  Let
  $(\bm{u}_h, \bm{p}_h) \in \bm{X}_h^v \times \bm{Q}_h^p$ solve
  \cref{eq:stokes_wf_discrete}.  Then
  \begin{align}
    \begin{split} \label{eq:p_estimate}
      \norm{\pi_h p -  p_h}
      \lesssim \Big\{
      \nu \inf_{\bm{v}_h \in \bm{X}_h^v}
      \Big[
      \vnorm{\bm{u} - \bm{v}_h}
      + \gsnorm{v_h}
      \Big]
      + \osc{\mathbb{P}f}
      \Big\}.
    \end{split}
  \end{align}
\end{theorem}
\begin{proof}
  Set $r_h := (\pi_h p - p_h) \in Q_h$.  We utilize the auxiliary
  inf-sup condition established in \cite[Lemma
  5]{rhebergen2020embedded}, which tells us that
  \begin{align}
    \label{eq:pressure_inf_sup}
    \norm{r_h} \lesssim
    \sup_{\bm{w}_h \in X_h^{\mathrm{BDM}} \times (\bar{X}_h \cap
    C^0(\Gamma_0)^d)}
    \frac{(r_h, \nabla \cdot w_h)}{\vnorm{\bm{w}_h}}.
  \end{align}
  Consider
  $\bm{w}_h \in X_h^{\mathrm{BDM}} \times (\bar{X}_h \cap C^0(\Gamma_0)^d)$.
  Then $\jump{w_h}|_F \cdot n_F = 0$ for all $F \in \mathcal{F}_h$ so that
  \begin{equation}
    \label{eq:pressure_b_1}
    -(p_h, \nabla \cdot w_h) = b_h(w_h, \bm{p}_h)
    = (f, w_h) - \nu a_h(\bm{u}_h, \bm{w}_h).
  \end{equation}
  On the other hand, since
  $\nabla \cdot w_h = \nabla \cdot E_h w_h \in Q_h$ we have
  \begin{align}
    \label{eq:pressure_b_2}
    \begin{split}
      (\pi_h p, \nabla \cdot w_h) &= (p, \nabla \cdot E_h w_h) \\
      &= -b(E_h w_h, p) \\
      &= -(f, E_h w_h) + \nu a(u, E_h w_h) \\
      &= -(f, E_h w_h) + \nu a_h(\bm{u}, (E_h w_h, E_h w_h)).
    \end{split}
  \end{align}
  Set $\bm{z}_h = \bm{w}_h - (E_h w_h, E_h w_h)$.  Combining
  \cref{eq:pressure_b_1} and \cref{eq:pressure_b_2} we obtain
  \begin{align}
    \label{eq:pressure_three_terms}
    \begin{split}
      (r_h, \nabla \cdot w_h)
      = \ & (f, w_h - E_h w_h) \\
      \ & - \nu a_h(\bm{u}_h, \bm{w}_h)
      + \nu a_h(\bm{u}, (E_h w_h, E_h w_h)) \\
      = \ & \underbrace{(f, w_h - E_h w_h)}_{I_1} \\
      \ & + \nu \underbrace{a_h(\bm{u} - \bm{u}_h, (E_h w_h, E_h w_h))}_{I_2}
      - \nu \underbrace{a_h(\bm{u}_h, \bm{z}_h)}_{I_3}.
    \end{split}
  \end{align}
  Since $w_h \in X_h^{\mathrm{BDM}}$ and $\nabla \cdot (w_h - E_h w_h) = 0$ we
  have that $(w_h - E_h w_h) \in L_{\sigma}^2(\Omega)$ 
  (recall \cref{eqn:Lsigma}).
  As a result,
  $I_1 = (\mathbb{P} f, w_h - E_h w_h)$.  Applying the Cauchy--Schwarz
  inequality and \cref{itm:e_h_interp} of \cref{lemma:enrichment} with $k=0$
  therefore yields
  \begin{equation}
    \label{eq:pressure_I_1_bound}
    \abs{I_1} \lesssim \osc{\mathbb{P} f} \jsnorm{w_h} \lesssim \osc{\mathbb{P} f} \vnorm{\bm{w}_h}.
  \end{equation}
  A bound for $\abs{I_2}$ follows from \cref{lemma:a_h_boundedness}
  and \cref{itm:e_h_bounded} of \cref{lemma:enrichment}:
  \begin{align}
    \label{eq:pressure_I_2_bound}
    \begin{split}
      \abs{I_2} &\lesssim \vnorm{\bm{u} - \bm{u}_h} \vnorm{(E_h w_h, E_h w_h)} \\
      &= \vnorm{\bm{u} - \bm{u}_h} \dgnorm{E_h w_h} \\
      &\lesssim \vnorm{\bm{u} - \bm{u}_h} \vnorm{\bm{w}_h}.
    \end{split}
  \end{align}
  Also, the same arguments used in \cref{lemma:a_h_consistency_error}
  show that
  \begin{equation}
    \label{eq:pressure_I_3_bound_1}
    \abs{I_3} \lesssim \Big[\vnorm{\bm{u} - \bm{u}_h} + \gsnorm{u_h} \Big] \vnorm{\bm{w}_h}.
  \end{equation}
  But for any $\bm{v}_h \in \bm{X}_h^v$, the triangle inequality and a
  discrete trace inequality yields
  \begin{align}
    \label{eq:pressure_I_3_bound_2}
    \begin{split}
      \gsnorm{u_h} &\leq \gsnorm{u_h - v_h} + \gsnorm{v_h} \\
      &\lesssim \norm{\nabla_h(u_h - v_h)} + \gsnorm{v_h} \\
      &\leq \vnorm{\bm{u}_h - \bm{v}_h} + \gsnorm{v_h} \\
      &\leq \vnorm{\bm{u} - \bm{u}_h} + \Big[ \vnorm{\bm{u} - \bm{v}_h} + \gsnorm{v_h} \Big].
    \end{split}
  \end{align}
  Combining \cref{eq:pressure_I_3_bound_1} and
  \cref{eq:pressure_I_3_bound_2} gives
  \begin{align}
    \label{eq:pressure_I_3_bound}
    \abs{I_3} \lesssim \Big\{
    \vnorm{\bm{u} - \bm{u}_h}
    + \inf_{\bm{v}_h \in \bm{X}_h^v}
    \Big[
    \vnorm{\bm{u} - \bm{v}_h}
    + \gsnorm{v_h}
    \Big] \Big\} \vnorm{\bm{w}_h}.
  \end{align}
  Inserting the bounds
  \cref{eq:pressure_I_1_bound,eq:pressure_I_2_bound,eq:pressure_I_3_bound}
  into \cref{eq:pressure_three_terms}, and using
  \cref{thm:main_u_estimate} to bound $\vnorm{\bm{u} - \bm{u}_h}$, we
  obtain
  \begin{align}
    \label{eq:pressure_three_terms_done}
    (r_h, \nabla \cdot w_h) \lesssim
    \Big\{
    \nu \inf_{\bm{v}_h \in \bm{X}_h^v}
    \Big[
    \vnorm{\bm{u} - \bm{v}_h}
    + \gsnorm{v_h}
    \Big]
    + \osc{\mathbb{P} f}
    \Big\} \vnorm{\bm{w}_h}.
  \end{align}
  Finally, combining \cref{eq:pressure_three_terms_done} and the
  inf-sup condition \cref{eq:pressure_inf_sup} we get
  \begin{align*}
    \norm{r_h}
    \lesssim
    \Big\{
    \nu \inf_{\bm{v}_h \in \bm{X}_h^v}
    \Big[
    \vnorm{\bm{u} - \bm{v}_h}
    + \gsnorm{v_h}
    \Big]
    + \osc{\mathbb{P} f}
    \Big\},
  \end{align*}
  which is the desired result.  \qed
\end{proof}

\begin{corollary}[Pressure convergence rate] \label{coro:pressure_rate}
In addition to the assumptions of \cref{thm:pressure_error}, assume that
$(u, p) \in H^{1+s}(\Omega)^d \times H^{s}(\Omega)$ for some
$s \in [0, 1]$.  Then
\begin{equation} \label{eq:pressure_rate}
	\norm{p - p_h} \lesssim_s
	h^{s} \norm{p}_{s} +
	\nu h^{s} \norm{u}_{1 + s} + h \norm{\mathbb{P} f}.
\end{equation}
\end{corollary}
\begin{proof}
By standard approximation properties of the $L^2$-orthogonal projector,
\begin{equation} \label{eq:pressure_rate_1}
	\norm{p - \pi_h p} \lesssim_s h^{s} \norm{p}_{s}.
\end{equation}
On the other hand, combining \cref{thm:pressure_error} and
\cref{lemma:interp_error} we find that
\begin{equation}  \label{eq:pressure_rate_2}
	\norm{\pi_h p -  p_h} \lesssim_s \nu h^{s} \norm{u}_{1+s}
	+ h \norm{\mathbb{P} f}.
\end{equation}
Using the bounds \cref{eq:pressure_rate_1} and \cref{eq:pressure_rate_2}
in the decomposition \cref{eq:pressure_error_decomp} yields the desired result
in \cref{eq:pressure_rate}. \qed
\end{proof}

\section{Numerical examples}
\label{sec:numerical_examples}

In this section we support our theoretical findings with numerical
examples.  Strictly speaking, the examples that we consider are outside
of the scope of our theory, because they involve inhomogenous Dirichlet
boundary conditions. Nevertheless, we will see that our numerical
observations agree with the theoretical predictions of
\cref{sec:error_analysis}.

All numerical examples have been implemented in NGSolve \cite{ngsolve}.
The penalty parameter is taken as $\alpha = 6k^2$ where $k$ is the polynomial
degree of the velocity finite element space.
We discuss numerical results only for the EDG--HDG method; our findings
for the HDG method are very similar in all cases.

\subsection{Convergence under minimal regularity}
\label{sec:numerical_examples_mr}

We consider the Stokes problem on the unit square $\Omega = (0, 1)^2$
with $f = 0$ and $\nu = 1$.  We impose Dirichlet boundary conditions on
the discrete solution by interpolating the exact solution.  The exact
solution is taken from \cite[Example~4]{verfurth1989posteriori} and in
polar coordinates is given by
\begin{equation} \label{eq:ex_mr_sln}
  u = \frac{3}{2} \sqrt{r}
  \begin{bmatrix}
    \cos\left(\frac{\theta}{2}\right) - \cos\left(\frac{3\theta}{2}\right)
    \\[0.5em]
    3 \sin\left(\frac{\theta}{2}\right) - \sin\left(\frac{3 \theta}{2}\right)
  \end{bmatrix}, \quad
  p = -6r^{-1/2} \cos\left(\frac{\theta}{2}\right).
\end{equation}
We note that $(u, p) \in H^{1+s}(\Omega)^d \times H^s(\Omega)$ for all
$0 \leq s < 1/2$.

The computed velocity and pressure errors for the EDG--HDG method
using the lowest-order $P^1-P^0$ discretization and the $P^2-P^1$
discretization are shown in \cref{tab:min_reg_example}. Both
discretizations are seen to converge at the same rate.  The velocity
error in the discrete $H^1$-norm and the pressure error in the
$L^2$-norm are observed to converge as roughly $h^{1/2}$.  This is
consistent with the regularity of the exact solution and the
predictions of \cref{coro:conv_result} and
\cref{coro:pressure_rate}. Finally, the velocity error in the
$L^2$-norm is observed to converge as roughly $h^{3/2}$.  Because
$\Omega$ is convex and therefore \cref{asm:dual_reg} holds with
$s_0=1$ (see e.g.~\cite{kellogg1976regularity}), this observed
convergence rate is consistent with \cref{thm:v_conv_l2}.

\begin{table}
  \caption{ Computed errors for the minimal regularity test case of
    \cref{sec:numerical_examples_mr} using the EDG--HDG method with
    different polynomial orders.  In all cases, the discrete velocity
    solution is divergence-free up to machine precision.  }
  \begin{center}
    \begin{tabular}{cc|cc|cc|cc}
      \hline
      &&&\\[-1em] 	
      Degree & Cells & $\norm{u-u_h}$ & Rate & $\vnorm{\bm{u}-\bm{u}_h}$ & Rate
                                             & $\norm{p-p_h}$ & Rate \\
      \hline
      &&&\\[-1em] 	
      $P^1$--$P^0$
      &     24 & 7.2e-02 &   - & 1.5e+00 &   - & 5.8e+00 &   - \\
      &     96 & 2.2e-02 & 1.7 & 8.1e-01 & 0.9 & 1.2e+00 & 2.3 \\
      &    384 & 7.6e-03 & 1.5 & 5.9e-01 & 0.5 & 8.2e-01 & 0.5 \\
      &   1536 & 2.8e-03 & 1.5 & 4.2e-01 & 0.5 & 5.8e-01 & 0.5 \\
      &   6144 & 9.8e-04 & 1.5 & 3.0e-01 & 0.5 & 4.1e-01 & 0.5 \\
      \multicolumn{8}{c}{} \\
      $P^2$--$P^1$
      &     24 & 2.8e-02 &   - & 8.4e-01 &   - & 1.4e+00 &   - \\
      &     96 & 7.6e-03 & 1.9 & 4.0e-01 & 1.1 & 5.2e-01 & 1.4 \\
      &    384 & 2.7e-03 & 1.5 & 2.9e-01 & 0.5 & 3.7e-01 & 0.5 \\
      &   1536 & 9.5e-04 & 1.5 & 2.0e-01 & 0.5 & 2.6e-01 & 0.5 \\
      &   6144 & 3.4e-04 & 1.5 & 1.4e-01 & 0.5 & 1.8e-01 & 0.5 \\
      \hline
    \end{tabular}
    \label{tab:min_reg_example}
  \end{center}
\end{table}

\subsection{Pressure-robust velocity approximation}
\label{sec:numerical_examples_pr}

To demonstrate pressure-robustness in the minimal regularity setting, we
consider a Stokes problem, taken from
\cite[Example~3]{verfurth1989posteriori}, on the L-shaped domain $\Omega
= (-1, 1)^2 \setminus ([0, 1] \times [-1, 0])$ and we vary the viscosity
$\nu$.  Consider
\begin{multline*}
  \psi(\theta) =
    \frac{1}{1 + \lambda} \sin((1 + \lambda) \theta) \cos(\lambda \omega) -
    \cos((1 + \lambda) \theta)
  \\
    - \frac{1}{1 - \lambda} \sin((1 - \lambda) \theta) \cos(\lambda \omega) +
    \cos((1 - \lambda) \theta),
\end{multline*}
and let $\lambda = 856399/1572864 \approx 0.54$ and $\omega = 3\pi/2$.
Our exact solution $(u, p)$ is given in polar coordinates by
\begin{equation*}
  u = r^{\lambda}
  \begin{bmatrix}
    (1 + \lambda) \sin(\theta) \psi(\theta) + \cos(\theta) \psi'(\theta) \\
    -(1 + \lambda) \cos(\theta) \psi(\theta) + \sin(\theta) \psi'(\theta)
  \end{bmatrix}, \quad
  p = \nu p_1 + p_2,
\end{equation*}
where
\begin{equation*}
  p_1 = r^{\lambda - 1} ((1 + \lambda)^2 \psi'(\theta) + \psi'''(\theta)) / (1 - \lambda), \quad
  p_2 = x^3 + y^3.
\end{equation*}
Note that $-\nabla^2 u + \nabla p_1 = 0$ and therefore
$-\nu \nabla^2 u + \nabla p = f$ where $f = \nabla p_2$.  Also, there
holds $(u, p) \in H^{1+s}(\Omega)^d \times H^s(\Omega)$ for all
$0 \leq s < \lambda$.

We compare the lowest-order EDG--HDG method to the lowest-order EDG
method of \cite{rhebergen2020embedded} (see also \cite{labeur:2012} on
the EDG method).  The EDG method, which uses a continuous facet finite
element space for both the velocity and pressure, is not
pressure-robust~\cite{rhebergen2020embedded}.  We set the viscosity to
be either $\nu = 1$ or $\nu = 10^{-5}$.  The computed velocity errors
for this example are shown in \cref{fig:pressure_robust_example}.

\begin{figure}
  \includegraphics{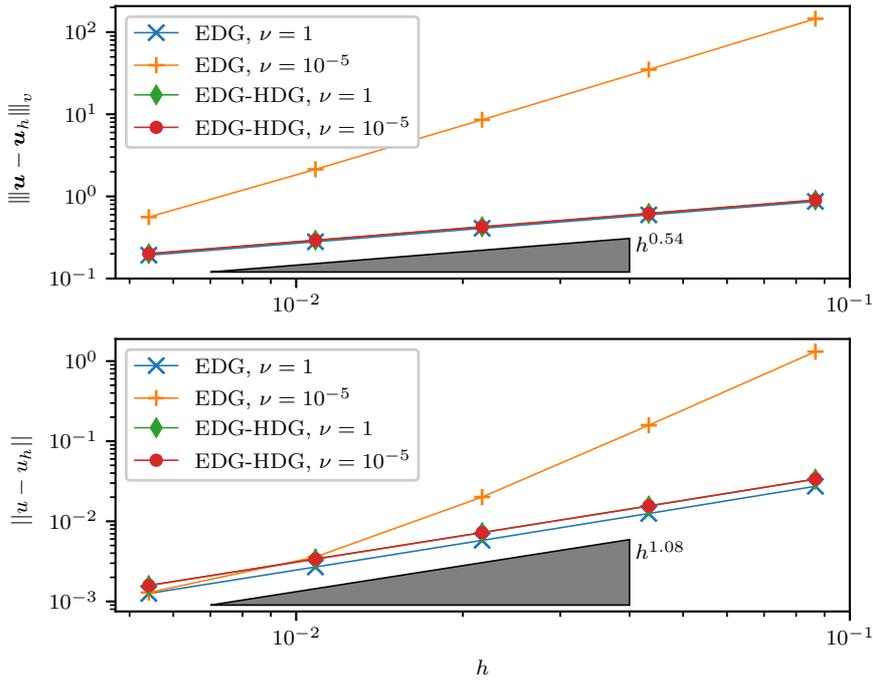}
  \caption{ Computed velocity errors for the pressure-robustness test
    case of \cref{sec:numerical_examples_pr}.  We compare the
    lowest-order EDG and EDG--HDG methods with $\nu=1$ and
    $\nu=10^{-5}$. }
  \label{fig:pressure_robust_example}
\end{figure}

For the EDG--HDG method, the velocity error is observed to be
independent of the viscosity, confirming pressure-robustness.  The
velocity error for this method converges in the discrete $H^1$-norm as
roughly $h^{0.54}$.  This is consistent with the regularity of $u$ and
\cref{coro:conv_result}.  Furthermore, according to
\cite[Section 5]{chorfi2014geometric}, on this domain \cref{asm:dual_reg} holds
with $s_0 \approx 0.54$.  Therefore, \cref{thm:v_conv_l2} predicts the velocity
error in the $L^2$-norm to converge as roughly $(h^{0.54})^2 = h^{1.08}$, which
is consistent with the empirical convergence rates displayed in
\cref{fig:pressure_robust_example}.

When $\nu = 1$ the velocity error for the EDG method is comparable to
that of the EDG--HDG method.  However, when $\nu = 10^{-5}$ the velocity
error for the EDG method increases substantially, at least in the regime
of large~$h$. In this regime, we hypothesize that the velocity error for
the EDG method is dominated by the pressure best approximation error
scaled by the inverse viscosity.  Recalling that $p = \nu p_1 + p_2$, we
can estimate the pressure best approximation error as
\begin{align}
  \label{eq:pbe_estimate}
  \begin{split}
    \inf_{q_h \in Q_h} \norm{p - q_h} &\leq \nu \inf_{q_h \in Q_h} \norm{p_1 - q_h}
    + \inf_{q_h \in Q_h} \norm{p_2 - q_h}
    \\
    &\lesssim_s \nu h^s \norm{p_1}_{s} + h^1 \norm{p_2}_{1},
  \end{split}
\end{align}
for any $0 \leq s < \lambda$.  For $h$ sufficiently large
\cref{eq:pbe_estimate} converges pre-asymptotically at a rate of $h^1$,
while the asymptotic convergence rate of \cref{eq:pbe_estimate}
is~$h^s$.  This behavior appears to be reflected in
\cref{fig:pressure_robust_example}, where for $\nu = 10^{-5}$ the
velocity error of the EDG method pre-asymptotically converges at a
faster rate than the EDG--HDG method.

\subsection{Domain with a crack}
\label{sec:numerical_examples_crack}

\begin{table}
	\caption{Computed errors for the cracked domain test case of
		\cref{sec:numerical_examples_crack} using the lowest-order EDG--HDG
		method.  In all cases, the discrete velocity
		solution is divergence-free up to machine precision.  }
	\begin{center}
		\begin{tabular}{c|cc|cc|cc}
			\hline
			&&&\\[-1em] 	
			Cells & $\norm{u-u_h}$ & Rate & $\vnorm{\bm{u}-\bm{u}_h}$
			& Rate
			& $\norm{p-p_h}$ & Rate \\
			\hline
			&&&\\[-1em] 	
			1680 & 2.0e-03 & - & 4.5e-01 & - & 5.8e-01 & - \\
			6720 & 1.0e-03 & 1.0 & 3.2e-01 & 0.5 & 3.6e-01 & 0.7 \\
			26880 & 5.0e-04 & 1.0 & 2.3e-01 & 0.5 & 2.4e-01 & 0.6 \\
			107520 & 2.5e-04 & 1.0 & 1.6e-01 & 0.5 & 1.6e-01 & 0.6 \\
			430080 & 1.2e-04 & 1.0 & 1.1e-01 & 0.5 & 1.1e-01 & 0.5 \\
			\hline
		\end{tabular}
		\label{tab:crack_example}
	\end{center}
\end{table}

We consider the Stokes problem on $\Omega = (-1/10, 1/10)^2 \setminus
([0, 1/10) \times \{0\})$ with $f = 0$ and $\nu = 1$.  Notice that
$\Omega$ has a crack along the positive $x$-axis.  We use the same exact
solution from \cref{eq:ex_mr_sln}.  The computed velocity and pressure
errors for the lowest-order EDG--HDG method are shown in
\cref{tab:crack_example}.

The velocity error in the discrete $H^1$-norm and the pressure error in
the $L^2$-norm eventually both converge as roughly $h^{1/2}$.  This is
consistent with the regularity of the exact solution and the
predictions of \cref{coro:conv_result} and \cref{coro:pressure_rate}.
Furthermore, according to \cite[Section 5]{chorfi2014geometric}, on this domain
\cref{asm:dual_reg} holds for any $s_0 < 1/2$.
Therefore, \cref{thm:v_conv_l2} predicts the velocity
error in the $L^2$-norm to converge as roughly $h^{1}$, which is consistent
with the empirical convergence rate seen in \cref{tab:crack_example}.

\section{Conclusions}
\label{sec:conclusions}

We have analyzed two lowest-order hybridizable DG methods for the Stokes
problem, while assuming only $H^{1+s}$-regularity of the exact velocity
solution for any $s \in [0, 1]$.  A salient feature of our analysis is 
that it allows for the case of a domain with cracks.  
The key ingredient in our analysis is a suitable upper
bound on the consistency error of the hybridizable DG methods,
which we have derived by means of a divergence-preserving enrichment operator.
Our resultant error estimates for the velocity
are pressure-robust and optimal in the discrete energy norm.  We also
obtained an error bound for the pressure that is dependent on the
velocity only.
Our theoretical findings are supported by various numerical examples.

\section*{Statements and Declarations}

SR gratefully acknowledges support from the Natural Sciences and
Engineering Research Council of Canada through the Discovery Grant
program (RGPIN-05606-2015).

\bibliographystyle{spmpsci}
\providecommand{\urlprefix}{}	
\bibliography{references}

\begin{thebibliography}{10}
\providecommand{\url}[1]{{#1}}
\providecommand{\urlprefix}{URL }
\expandafter\ifx\csname urlstyle\endcsname\relax
  \providecommand{\doi}[1]{DOI~\discretionary{}{}{}#1}\else
  \providecommand{\doi}{DOI~\discretionary{}{}{}\begingroup
  \urlstyle{rm}\Url}\fi

\bibitem{arnold1984stable}
Arnold, D.N., Brezzi, F., Fortin, M.: A stable finite element for the {Stokes}
  equations.
\newblock Calcolo \textbf{21}(4), 337--344 (1984).
\newblock \urlprefix\url{https://doi.org/10.1007/BF02576171}

\bibitem{badia2014error}
Badia, S., Codina, R., Gudi, T., Guzm{\'a}n, J.: Error analysis of
  discontinuous {Galerkin} methods for the {Stokes} problem under minimal
  regularity.
\newblock IMA J. Numer. Anal. \textbf{34}(2), 800--819 (2013).
\newblock \urlprefix\url{https://doi.org/10.1093/imanum/drt022}

\bibitem{bernardi1985analysis}
Bernardi, C., Raugel, G.: Analysis of some finite elements for the {Stokes}
  problem.
\newblock Math. Comput. \textbf{44}(169), 71--79 (1985).
\newblock \urlprefix\url{https://doi.org/10.2307/2007793}

\bibitem{boffi2013mixed}
Boffi, D., Brezzi, F., Fortin, M.: {Mixed Finite Element Methods and
  Applications}, \emph{{Springer Series in Computational Mathematics}},
  vol.~44.
\newblock Springer, Heidelberg (2013).
\newblock \urlprefix\url{https://doi.org/10.1007/978-3-642-36519-5}

\bibitem{brenner2008mathematical}
Brenner, S.C., Scott, L.R.: {The Mathematical Theory of Finite Element
  Methods}, \emph{{Texts in Applied Mathematics}}, vol.~15, 3rd edn.
\newblock Springer, New York (2008).
\newblock \urlprefix\url{https://doi.org/10.1007/978-0-387-75934-0}

\bibitem{chorfi2014geometric}
Chorfi, N.: Geometric singularities of the {Stokes} problem.
\newblock Abstr. Appl. Anal. \textbf{2014}, Article ID 491326 (2014).
\newblock \urlprefix\url{http://dx.doi.org/10.1155/2014/491326}

\bibitem{cockburn2009unified}
Cockburn, B., Gopalakrishnan, J., Lazarov, R.: Unified hybridization of
  discontinuous {Galerkin}, mixed, and continuous {Galerkin} methods for second
  order elliptic problems.
\newblock SIAM J. Numer. Anal. \textbf{47}(2), 1319--1365 (2009).
\newblock \urlprefix\url{https://doi.org/10.1137/070706616}

\bibitem{cockburn2007note}
Cockburn, B., Kanschat, G., Sch{\"o}tzau, D.: A note on discontinuous
  {Galerkin} divergence-free solutions of the {Navier}--{Stokes} equations.
\newblock J. Sci. Comput. \textbf{31}(1-2), 61--73 (2007).
\newblock \urlprefix\url{https://doi.org/10.1007/s10915-006-9107-7}

\bibitem{cockburn2014divergence}
Cockburn, B., {Sayas, F.-J.}: Divergence-conforming {HDG} methods for {Stokes}
  flows.
\newblock Math. Comput. \textbf{83}(288), 1571--1598 (2014).
\newblock \urlprefix\url{https://doi.org/10.1090/S0025-5718-2014-02802-0}

\bibitem{dauge_stokesreg1}
Dauge, M.: {Elliptic Boundary Value Problems on Corner Domains}, \emph{{Lecture
  Notes in Mathematics}}, vol. 1341.
\newblock Springer, Berlin (1988).
\newblock \urlprefix\url{https://doi.org/10.1007/BFb0086682}

\bibitem{dauge_stokesreg2}
Dauge, M.: Stationary {Stokes} and {Navier}--{Stokes} systems on two- or
  three-dimensional domains with corners. {Part I}. {Linearized} equations.
\newblock SIAM J. Math. Anal. \textbf{20}(1), 74--97 (1989).
\newblock \urlprefix\url{https://doi.org/10.1137/0520006}

\bibitem{di2011mathematical}
{Di Pietro, D.A.}, Ern, A.: {Mathematical Aspects of Discontinuous Galerkin
  Methods}, \emph{{Mathématiques et Applications}}, vol.~69.
\newblock Springer, Berlin (2012).
\newblock \urlprefix\url{https://doi.org/10.1007/978-3-642-22980-0}

\bibitem{dupont1980polynomial}
Dupont, T., Scott, R.: Polynomial approximation of functions in {Sobolev}
  spaces.
\newblock Math. Comput. \textbf{34}(150), 441--463 (1980).
\newblock \urlprefix\url{https://doi.org/10.1090/S0025-5718-1980-0559195-7}

\bibitem{ern2013theory}
Ern, A., {Guermond, J.-L.}: {Theory and Practice of Finite Elements},
  \emph{{Applied Mathematical Sciences}}, vol. 159.
\newblock Springer, New York (2004).
\newblock \urlprefix\url{https://doi.org/10.1007/978-1-4757-4355-5}

\bibitem{ern2017finitequasi}
Ern, A., {Guermond, J.-L.}: Finite element quasi-interpolation and best
  approximation.
\newblock ESAIM Math. Model. Numer. Anal. \textbf{51}(4), 1367--1385 (2017).
\newblock \urlprefix\url{https://doi.org/10.1051/m2an/2016066}

\bibitem{fu2019explicit}
Fu, G.: An explicit divergence-free {DG} method for incompressible flow.
\newblock Comput. Methods Appl. Mech. Eng. \textbf{345}, 502--517 (2019).
\newblock \urlprefix\url{https://doi.org/10.1016/j.cma.2018.11.012}

\bibitem{guzey2007embedded}
G{\"u}zey, S., Cockburn, B., {Stolarski, H.K.}: The embedded discontinuous
  {Galerkin} method: application to linear shell problems.
\newblock Int. J. Numer. Methods Eng. \textbf{70}(7), 757--790 (2007).
\newblock \urlprefix\url{https://doi.org/10.1002/nme.1893}

\bibitem{guzman2014conforming}
Guzm{\'a}n, J., Neilan, M.: Conforming and divergence-free {Stokes} elements on
  general triangular meshes.
\newblock Math. Comput. \textbf{83}(285), 15--36 (2014).
\newblock \urlprefix\url{https://doi.org/10.1090/S0025-5718-2013-02753-6}

\bibitem{guzman_neilan_3d_stokes}
Guzmán, J., Neilan, M.: {Conforming and divergence-free {Stokes} elements in
  three dimensions}.
\newblock IMA J. Numer. Anal. \textbf{34}(4), 1489--1508 (2014).
\newblock \urlprefix\url{https://doi.org/10.1093/imanum/drt053}

\bibitem{john2017divergence}
John, V., Linke, A., Merdon, C., Neilan, M., Rebholz, L.G.: On the divergence
  constraint in mixed finite element methods for incompressible flows.
\newblock SIAM Rev. \textbf{59}(3), 492--544 (2017).
\newblock \urlprefix\url{https://doi.org/10.1137/15M1047696}

\bibitem{karakashian2003posteriori}
Karakashian, O.A., Pascal, F.: A posteriori error estimates for a discontinuous
  {Galerkin} approximation of second-order elliptic problems.
\newblock SIAM J. Numer. Anal. \textbf{41}(6), 2374--2399 (2003).
\newblock \urlprefix\url{https://doi.org/10.1137/S0036142902405217}

\bibitem{kellogg1976regularity}
{Kellogg, R.B.}, {Osborn, J.E.}: A regularity result for the {Stokes} problem
  in a convex polygon.
\newblock J. Funct. Anal. \textbf{21}(4), 397--431 (1976).
\newblock \urlprefix\url{https://doi.org/10.1016/0022-1236(76)90035-5}

\bibitem{moment_div_preserving}
Kreuzer, C., Verf{\"u}rth, R., Zanotti, P.: Quasi-optimal and pressure robust
  discretizations of the {Stokes} equations by moment- and
  divergence-preserving operators.
\newblock Comput. Methods Appl. Math. \textbf{21}(2), 423--443 (2021).
\newblock \urlprefix\url{https://doi.org/10.1515/cmam-2020-0023}

\bibitem{kreuzer2020quasi}
Kreuzer, C., Zanotti, P.: Quasi-optimal and pressure-robust discretizations of
  the {Stokes} equations by new augmented {Lagrangian} formulations.
\newblock IMA J. Numer. Anal. \textbf{40}(4), 2553--2583 (2019).
\newblock \urlprefix\url{https://doi.org/10.1093/imanum/drz044}

\bibitem{labeur:2012}
Labeur, R.J., Wells, G.N.: Energy stable and momentum conserving hybrid finite
  element method for the incompressible {N}avier--{S}tokes equations.
\newblock SIAM J. Sci. Comput. \textbf{34}(2), A889--A913 (2012).
\newblock \urlprefix\url{http://dx.doi.org/10.1137/100818583}

\bibitem{lehrenfeld2016high}
Lehrenfeld, C., Sch{\"o}berl, J.: High order exactly divergence-free hybrid
  discontinuous {Galerkin} methods for unsteady incompressible flows.
\newblock Comput. Methods Appl. Mech. Eng. \textbf{307}, 339--361 (2016).
\newblock \urlprefix\url{https://doi.org/10.1016/j.cma.2016.04.025}

\bibitem{li2014new}
Li, M., Mao, S., Zhang, S.: New error estimates of nonconforming mixed finite
  element methods for the {Stokes} problem.
\newblock Math. Methods Appl. Sci. \textbf{37}(7), 937--951 (2014).
\newblock \urlprefix\url{https://doi.org/10.1002/mma.2849}

\bibitem{linke2019pressure}
Linke, A., Merdon, C., Neilan, M.: Pressure-robustness in quasi-optimal a
  priori estimates for the {Stokes} problem.
\newblock Electron. Trans. Numer. Anal. \textbf{52}, 281--294 (2020).
\newblock \urlprefix\url{https://doi.org/10.1553/etna_vol52s281}

\bibitem{linke2018quasi}
Linke, A., Merdon, C., Neilan, M., Neumann, F.: Quasi-optimality of a
  pressure-robust nonconforming finite element method for the {Stokes}-problem.
\newblock Math. Comput. \textbf{87}(312), 1543--1566 (2018).
\newblock \urlprefix\url{https://doi.org/10.1090/mcom/3344}

\bibitem{rhebergen2017analysis}
Rhebergen, S., Wells, G.N.: Analysis of a hybridized/interface stabilized
  finite element method for the {Stokes} equations.
\newblock SIAM J. Numer. Anal. \textbf{55}(4), 1982--2003 (2017).
\newblock \urlprefix\url{https://doi.org/10.1137/16M1083839}

\bibitem{rhebergen2018hybridizable}
Rhebergen, S., Wells, G.N.: A hybridizable discontinuous {Galerkin} method for
  the {Navier}--{Stokes} equations with pointwise divergence-free velocity
  field.
\newblock J. Sci. Comput. \textbf{76}(3), 1484--1501 (2018).
\newblock \urlprefix\url{https://doi.org/10.1007/s10915-018-0671-4}

\bibitem{rhebergen2020embedded}
Rhebergen, S., Wells, G.N.: An embedded--hybridized discontinuous {Galerkin}
  finite element method for the {Stokes} equations.
\newblock Comput. Methods Appl. Mech. Eng. \textbf{358}, 112619 (2020).
\newblock \urlprefix\url{https://doi.org/10.1016/j.cma.2019.112619}

\bibitem{ngsolve}
Sch{\"o}berl, J.: {C++11} implementation of finite elements in {NGSolve}.
\newblock {Technical Report, ASC Report 30/2014}, {Institute for Analysis and
  Scientific Computing}, {Vienna University of Technology} (2014).
\newblock
  \urlprefix\url{https://www.asc.tuwien.ac.at/~schoeberl/wiki/publications/ngs-cpp11.pdf}

\bibitem{scott1990finite}
Scott, L.R., Zhang, S.: Finite element interpolation of nonsmooth functions
  satisfying boundary conditions.
\newblock Math. Comput. \textbf{54}(190), 483--493 (1990).
\newblock \urlprefix\url{https://doi.org/10.1090/S0025-5718-1990-1011446-7}

\bibitem{veeser2017quasi_nonlipschitz}
Veeser, A., Zanotti, P.: Quasi-optimal nonconforming methods for second-order
  problems on domains with {non-Lipschitz} boundary.
\newblock In: F.A. Radu, K.~Kumar, I.~Berre, J.M. Nordbotten, I.S. Pop (eds.)
  Numerical Mathematics and Advanced Applications ENUMATH 2017, pp. 461--469.
  Springer, Cham (2019).
\newblock \urlprefix\url{https://doi.org/10.1007/978-3-319-96415-7_41}

\bibitem{verfurth1989posteriori}
Verf{\"u}rth, R.: A posteriori error estimators for the {Stokes} equations.
\newblock Numer. Math. \textbf{55}(3), 309--325 (1989).
\newblock \urlprefix\url{https://doi.org/10.1007/BF01390056}

\bibitem{verfurth2019quasi}
Verf{\"u}rth, R., Zanotti, P.: A quasi-optimal {Crouzeix--Raviart}
  discretization of the {Stokes} equations.
\newblock SIAM J. Numer. Anal. \textbf{57}(3), 1082--1099 (2019).
\newblock \urlprefix\url{https://doi.org/10.1137/18M1177688}

\bibitem{wang2007new}
Wang, J., Ye, X.: New finite element methods in computational fluid dynamics by
  {H(div)} elements.
\newblock SIAM J. Numer. Anal. \textbf{45}(3), 1269--1286 (2007).
\newblock \urlprefix\url{https://doi.org/10.1137/060649227}

\end{thebibliography}
\appendix
\section{Enrichment operator}
\label{sec:appendix_a}

We prove here \cref{lemma:enrichment} for the three-dimensional case
($d=3$). We shall use the angled bracket notation $\ip{\cdot}{\cdot}_D$;
this is simply used to denote the $L^2$-inner-product on a domain $D$
with dimension strictly less than~$d=3$.

For $K \in \mathcal{T}$ we recall from \cref{subsec:mesh} that
$\mathcal{F}_{K, h} \subset \mathcal{F}_h$ denotes the four faces of
$K$. Also, let $\mathcal{E}_K$ denote the six edges of $K$ and
$\mathcal{V}_K$ the four vertices of $K$.  The collection of all mesh
edges is written as $\mathcal{E}_h := \cup_{K \in \mathcal{T}}
\mathcal{E}_K = \mathcal{E}_b \cup \mathcal{E}_i$ where $\mathcal{E}_b$
denotes the boundary edges and $\mathcal{E}_i$ the interior edges.
Likewise, the collection of all mesh vertices is written as
$\mathcal{V}_h := \cup_{K \in \mathcal{T}} \mathcal{V}_K = \mathcal{V}_b
\cup \mathcal{V}_i$ where $\mathcal{V}_b$ denotes the boundary vertices
and $\mathcal{V}_i$ the interior vertices.  For an interior edge $e \in
\mathcal{E}_i$, we define the average of a function $v$ on $e$ as
\begin{equation*}
  \av{v}_e := \frac{1}{\abs{\mathcal{T}_e}}
  \sum_{\substack{K \in \mathcal{T}_e}} v_K|_e,
\end{equation*}
where $\mathcal{T}_e := \{K \in \mathcal{T} : e \in \mathcal{E}_K \}$
denotes the collection of elements having $e$ as an edge and $v_K :=
v|_K$. On boundary edges $e \in \mathcal{E}_b$, it will be convenient to
define $\av{v}_e := 0$.  Similarly, for an interior vertex $a \in
\mathcal{V}_i$, the average of a function $v$ on $a$ is defined as
\begin{equation*}
  \av{v}_a := \frac{1}{\abs{\mathcal{T}_a}}
  \sum_{\substack{K \in \mathcal{T}_a}} v_K(a),
\end{equation*}
where $\mathcal{T}_a := \{ K \in \mathcal{T} : a \in \mathcal{V}_K \}$
denotes the collection of elements having $a$ as a vertex.  On boundary
vertices $a \in \mathcal{V}_b$ it will be convenient to define $\av{v}_a
:= 0$.  Finally, throughout this proof we continue to use the definition
\crefrange{eq:av_def_interior}{eq:jump_av_def_boundary} of the average
operator on faces.

For $K \in \mathcal{T}$, let $V(K)$ denote the local three-dimensional
Guzm\'an--Neilan finite element space defined by
\cite[eq.~(3.9)]{guzman_neilan_3d_stokes}.  This space has the
properties \cite[Lemma~3.4]{guzman_neilan_3d_stokes}
\begin{alignat*}{2}
  &[\mathcal{P}_1(K)]^3 \subset V(K), \qquad
  &&V(K) \subset [W^{1, \infty}(K) \cap C^0(\bar{K})]^3, \\
  &\nabla \cdot V(K) \subset \mathcal{P}_0(K), \qquad
  &&V(K)|_{\partial K} \subset [\mathcal{P}_3(\partial K)]^3.
\end{alignat*}
Moreover, a set of unisolvent degrees of freedom for $v \in V(K)$ is
given by \cite[Theorem~3.5]{guzman_neilan_3d_stokes}
\begin{subequations}
  \label{eq:gn_dofs}
  \begin{alignat}{2}
    &v(a) \qquad
    &&\fra a \in \mathcal{V}_K, \\
    &\ip{v}{w}_e \qquad
    &&\fra e \in \mathcal{E}_K, w \in [\mathcal{P}_1(e)]^3, \\
    &\ip{v}{w}_F \qquad
    &&\fra F \in \mathcal{F}_{K,h}, w \in [\mathcal{P}_0(F)]^3.
  \end{alignat}
\end{subequations}

For $K \in \mathcal{T}$ we now define the local operator
$E_K : X_h^{\mathrm{BDM}} \rightarrow V(K)$ as follows.  For
$v_h \in X_h^{\mathrm{BDM}}$ we require that
\begin{subequations}
  \label{eq:local_E_defn}
  \begin{alignat}{2}
    &(E_K v_h)(a) = \av{v_h}_a \qquad
    &&\fra a \in \mathcal{V}_K, \\
    &\ip{E_K v_h - \av{v_h}_e}{w}_e = 0 \qquad
    &&\fra e \in \mathcal{E}_K, w \in [\mathcal{P}_1(e)]^3, \\
    &\ip{E_K v_h - \av{v_h}}{w}_F = 0 \qquad
    &&\fra F \in \mathcal{F}_{K,h} \cap \mathcal{F}_i, w \in
    [\mathcal{P}_0(F)]^3,
    \label{eq:local_E_defn_int_face} \\
    &\ip{E_K v_h}{w}_F = 0 \qquad
    &&\fra F \in \mathcal{F}_{K,h} \cap \mathcal{F}_b, w \in
    [\mathcal{P}_0(F)]^3.
  \end{alignat}
\end{subequations}
The degrees of freedom \cref{eq:gn_dofs} imply that the operator $E_K$
is well-defined.  We then define
$E_h : X_h^{\mathrm{BDM}} \rightarrow H_0^1(\Omega)^d$ by 
$(E_h v_h)|_K = E_K v_h$
for all $v_h \in X_h^{\mathrm{BDM}}$.  Utilizing the inclusion
$V(K)|_{\partial K} \subset [\mathcal{P}_3(\partial K)]^3$ one can show that
$\jump{E_h v_h}|_F = 0$ for all $F \in \mathcal{F}_h$, and thus
$E_h v_h \in H_0^1(\Omega)^d$ holds.  It remains to verify that $E_h$ satisfies
\crefrange{itm:e_h_facet_mean}{itm:e_h_bounded} from
\cref{lemma:enrichment}.

That \cref{itm:e_h_facet_mean} holds is an immediate consequence of
\cref{eq:local_E_defn_int_face}.  To prove \cref{itm:e_h_div}, consider
the space $\mathcal{P}_{0,h} :=
\{q_h \in L^2(\Omega) : q_h|_K \in \mathcal{P}_0(K) \
\forall K \in \mathcal{T} \}$ of piecewise constant functions.  Let $v_h
\in X_h^{\mathrm{BDM}}$ and $q_h \in \mathcal{P}_{0,h}$.  Then element-wise
integration by parts and \cref{eq:local_E_defn_int_face} shows that
\begin{align}
  \label{eq:e_h_ibp_argument}
  \begin{split}
    \int_{\Omega} (\nabla \cdot E_h v_h) q_h \dif x
    &= \sum_{F \in \mathcal{F}_i} \int_F (E_h v_h \cdot n_F) \jump{q_h} \dif s \\
    &= \sum_{F \in \mathcal{F}_i} \int_F (\av{v_h} \cdot n_F) \jump{q_h} \dif s \\
    &= \int_{\Omega} (\nabla \cdot v_h) q_h \dif x,
  \end{split}
\end{align}
where the last equality in \cref{eq:e_h_ibp_argument} follows from the
fact that $\jump{v_h}|_F \cdot n_F = 0$ for all $F \in \mathcal{F}_h$.
But \cref{itm:e_h_div} now follows from \cref{eq:e_h_ibp_argument} as
$\nabla \cdot E_h v_h, \nabla \cdot v_h \in \mathcal{P}_{0,h}$ and $q_h \in
\mathcal{P}_{0,h}$ is arbitrary.

To prove \cref{itm:e_h_interp}, fix $k \in \{0, 1\}$ and $v_h \in
X_h^{\mathrm{BDM}}$.  Consider $K \in \mathcal{T}$ and set $w_K := E_K
v_h, v_K = v_h|_K$ and $z_K = w_K - v_K$.  Since $v_K \in [\mathcal{P}_1(K)]^3
\subset V(K)$, there holds $z_K \in V(K)$.  A scaling argument utilizing
the degrees of freedom \cref{eq:gn_dofs} then shows that
\begin{align}
  \label{itm:e_h_interp_three_terms}
  \begin{split}
    h_K^{2(k-1)} \abs{z_K}_{k,K}^2 \lesssim &
    \sum_{a \in \mathcal{V}_K} h_K \norm{z_K(a)}_2^2
    + \sum_{e \in \mathcal{E}_K}
    \sup_{\substack{\kappa_h \in [\mathcal{P}_1(e)]^3
    \\
    \norm{\kappa_h}_e = 1}} \abs{\ip{z_K}{\kappa_h}_e}^2
    \\
    &+ \sum_{F \in \mathcal{F}_{K,h}} \frac{1}{h_K}
    \sup_{\substack{\kappa_h \in [\mathcal{P}_0(F)]^3
    \\
    \norm{\kappa_h}_F = 1}} \abs{\ip{z_K}{\kappa_h}_F}^2
    \\
    \leq
    &\underbrace{\sum_{a \in \mathcal{V}_K} h_K \norm{\av{v_h}_a - v_K(a)}_2^2}_{I_1}
    + \underbrace{\sum_{e \in \mathcal{E}_K} \norm{\av{v_h}_e - v_K}_e^2}_{I_2}
    \\
    &+ \underbrace{
      \sum_{\substack{F \in \mathcal{F}_{K,h} \\ F \in \mathcal{F}_i}}
      \frac{1}{h_K} \norm{\av{v_h} - v_K}_F^2
      + \sum_{\substack{F \in \mathcal{F}_{K,h} \\ F \in \mathcal{F}_b}}
      \frac{1}{h_K} \norm{v_K}_F^2
    }_{I_3}.
  \end{split}
\end{align}
Because $\partial \Omega$ has codimension one, every boundary vertex of
the mesh is contained in some boundary face of the mesh, and likewise
for boundary edges. As a result, the same arguments used in
\cite[Lemma~4.7]{linke2018quasi} show that
\begin{align}
  I_1 &\lesssim \sum_{a \in \mathcal{V}_K} \sum_{F \in \mathcal{F}_a} \frac{1}{h_F} \norm{\jump{v_h}}_F^2,
	\label{eq:e_h_I1_bound} \\
  I_2 &\lesssim \sum_{e \in \mathcal{E}_K} \sum_{F \in \mathcal{F}_e} \frac{1}{h_F} \norm{\jump{v_h}}_F^2,
	\label{eq:e_h_I2_bound}
\end{align}
where $\mathcal{F}_a \subset \mathcal{F}_h$ denotes the collection of
all mesh faces having $a$ as a vertex, and $\mathcal{F}_e \subset
\mathcal{F}_h$ denotes the collection of all mesh faces having $e$ as an
edge.  We note that, due to midpoint continuity of Crouzeix--Raviart
elements, there is no term analogous to $I_3$ in
\cite[Lemma~4.7]{linke2018quasi}.  Fortunately, it is easy to see that
we can bound $I_3$ by means of
\begin{align}
  \label{eq:e_h_I3_bound}
  I_3 \lesssim \sum_{F \in \mathcal{F}_{K,h}} \frac{1}{h_F}
  \norm{\jump{v_h}}_F^2.
\end{align}
Using the bounds \crefrange{eq:e_h_I1_bound}{eq:e_h_I3_bound} in
\cref{itm:e_h_interp_three_terms}, and summing over $K \in \mathcal{T}$,
one obtains
\begin{align*}
  \sum_{K \in \mathcal{T}} h_K^{2(k-1)} \abs{E_h v_h - v_h}_{k,K}^2
  \lesssim \sum_{F \in \mathcal{F}_h} \frac{1}{h_F} \norm{\jump{v_h}}_F^2
  = \jsnorm{v_h}^2,
\end{align*}
so that \cref{itm:e_h_interp} holds.  Lastly, \cref{itm:e_h_bounded}
follows from \cref{itm:e_h_interp} with $k=1$ and the triangle
inequality:
\begin{align*}
  \norm{\nabla E_h v_h}
  &\leq \Big( \sum_{K \in \mathcal{T}} \abs{v_h - E_h v_h}_{1, K}^2 \Big)^{1/2}
    + \Big( \sum_{K \in \mathcal{T}} \abs{v_h}_{1,K}^2 \Big)^{1/2} \\
  &\lesssim \jsnorm{v_h} + \Big( \sum_{K \in \mathcal{T}} \abs{v_h}_{1,K}^2 \Big)^{1/2} \\
  &\lesssim \dgnorm{v_h}.
\end{align*}
This completes the proof of \cref{lemma:enrichment}.  \qed

\end{document}